\documentclass[11pt]{article}
\usepackage{amssymb}
\usepackage{amsmath}
\usepackage{amsthm}
\usepackage{cleveref}
\usepackage{amsfonts}
\usepackage{graphicx}
\usepackage{placeins}
\usepackage{comment}
\usepackage{xfrac}
\usepackage{esint}
\usepackage{enumerate}
\usepackage{mathtools}
\usepackage[margin=1in]{geometry} 
\def\theequation{\thesection.\arabic{equation}}
\numberwithin{equation}{section}

\usepackage[dvipsnames]{color}

\newcommand{\tr}[1]{\mathrm{tr}\,{#1}}

\newcommand{\e}{\varepsilon}
\newcommand{\R}{\mathbb R}

\newcommand{\hn}{\mathcal{H}^{N-1}}
\newcommand{\ap}{\mathrm{ap}}

\newcommand{\beq}{\begin{equation}}
\newcommand{\eeq}{\end{equation}}

\newcommand{\dive}{\mathrm{div}\,}
\newcommand{\curl}{\mathrm{curl}\,}
\newcommand{\dz}{\partial_z}
\newcommand{\dx}{\partial_x}
\newcommand{\dy}{\partial_y}
\newcommand{\dxx}{\dx^2}
\newcommand{\dyy}{\dy^2}
\newcommand{\dxy}{\partial_{xy}}

\newcommand{\dZ}{\partial_\xi}
\newcommand{\dE}{\partial_\eta}
\newcommand{\dZZ}{\partial_\xi^2}
\newcommand{\dEE}{\partial_\eta^2}
\newcommand{\dEZ}{\dE\dZ}
\newcommand{\dZE}{\dZ\dE}
\newcommand{\dZz}{\dZ\dz}
\newcommand{\dzZ}{\dz\dZ}
\newcommand{\dEz}{\dE\dz}
\newcommand{\dzE}{\dz\dE}

\let\oldmp\mp
\renewcommand\mp{m^+}

\newcommand{\mm}{m^-}

\newcommand{\ex}{\mathbf{x}}

\newcommand{\Sig}{\Sigma_{\xi\eta}}
\newcommand{\ag}{AG^{3D}(\Omega)}
\newcommand{\ago}{AG^{3D}_0(\Omega)}
\newcommand{\ssubset}{\subset\joinrel\subset}
\newcommand{\mres}{\mathbin{\vrule height 1.6ex depth 0pt width
0.13ex\vrule height 0.13ex depth 0pt width 1.3ex}}
\newcommand{\parallelsum}{\mathbin{\!/\mkern-5mu/\!}}
\newcommand\restr[2]{{
  \left.\kern-\nulldelimiterspace 
  #1 
  \right|_{#2} 
  }}

\newtheorem{conjecture}{Conjecture}
\newtheorem{theorem}{Theorem}[section]

\newtheorem{proposition}[theorem]{Proposition}
\newtheorem{lemma}[theorem]{Lemma}
\newtheorem{definition}{Definition}[section]
\newtheorem{remark}[theorem]{Remark}

\begin{document}

\title{\bf Nonlinear approximation of 3D smectic liquid crystals: sharp lower bound and compactness}
\author{Michael Novack\thanks{%
Department of Mathematics, The University of Texas at Austin, Austin, TX, USA }
\and Xiaodong Yan\thanks{%
Department of Mathematics, The University of Connecticut, Storrs, CT, USA } }
\maketitle

\begin{abstract}
We consider the 3D smectic energy%
\begin{equation*}
\mathcal{E}_{\epsilon }\left( u\right) =\frac{1}{2}\int_{\Omega }\frac{1}{\varepsilon }%
\left( \dz u-\frac{(\dx u)^{2}+(\dy u)^{2}}{2}\right) ^{2}+\varepsilon \left(
\dxx u+\dyy u\right) ^{2}dx\,dy\,dz.
\end{equation*}%
The model contains as a special case the well-known 2D Aviles-Giga model. We prove a sharp lower bound on $\mathcal{E}_{\varepsilon}$ as $\e \to 0$ by introducing 3D analogues of the Jin-Kohn entropies \cite{JinKoh00}. The sharp bound corresponds to an equipartition of energy between the bending and compression strains and was previously demonstrated in the physics literature only when the {\color{black} approximate} Gaussian curvature of each smectic layer vanishes. Also, for $\varepsilon _{n}\rightarrow 0$ and an energy-bounded sequence $\{u_n \}$  with $\|\nabla u_n\|_{L^{\color{black} p}(\Omega)},\,\|\nabla u_n\|_{L^2(\partial \Omega)}\leq C$ {\color{black} for some $p>6$}, we obtain compactness of $\nabla u_{n}$ in $L^{2}$ assuming that $\Delta_{xy}u_{n}$ has constant sign for each $n$.
\end{abstract}


{\color{black}
\section{\color{black}Introduction}

\label{sec:intro}
{\color{black}In this article, we analyze the energies
\begin{equation}
\mathcal{E}_{\varepsilon }\left( u\right) =\frac{1}{2}\int_{\Omega }\left[ 
\frac{1}{\varepsilon }\left( \partial _{z}u-\frac{1}{2}\left\vert \nabla
_{\bot }u\right\vert ^{2}\right) ^{2}+\varepsilon \left( \Delta _{\bot
}u\right) ^{2}\right] d\mathbf{x}, \label{3denergy}
\end{equation}
which represent the free energy of a smectic-A liquid crystal; see \Cref{subsec:phys} for a discussion of the relevant physics literature regarding smectics. Here $\Omega\subset \mathbb{R}^{3}$ is a bounded domain with Lipschitz boundary, and ${\bf x}=(x,y,z)$. The subscript ``$\bot$" denotes the restriction to the $x,y$ variables of a differential operator or the projection from $\R^3$ to $\R^2$, so
\begin{equation*}
    \nabla_\bot u= (\partial_x u ,\partial_y u)\quad\textup{and}\quad \Delta_\bot u = \partial^2_x u +\partial^2_y u.
\end{equation*}
and if $m=(m_1,m_2,m_3) \in \R^3$,
\begin{equation*}
    m_\bot = m_1\hat{x}+m_2\hat{y} \quad \textup{and}\quad m=(m_\bot,m_3) \in \R^3.
\end{equation*}
Our main interest is the asymptotic behavior of energies \eqref{3denergy} as $\e\rightarrow 0$, which corresponds to the regime in which the intrinsic length scale $\varepsilon$, cf. \eqref{e definition}, is vanishingly small compared to $\Omega$. 

We 
prove the following main results when $\varepsilon \rightarrow 0$: \vspace{0.1cm}

\begin{itemize}
\item a lower bound, sharp when $\nabla u \in (BV\cap L^\infty)(\Omega;\mathbb{R}^3)$, on $\mathcal{E}_{\varepsilon }$ when $\varepsilon
\rightarrow 0$ (\Cref{3dlbd}, \Cref{3dupbd}), and
\item a compactness theorem for the gradients of a sequence with bounded
energies (\Cref{compactness}) satisfying some additional technical assumptions.
\end{itemize}
These results generalize the authors' previous work \cite{NY20} on the 2D model
\begin{equation}
\mathcal{J}_{\varepsilon }\left( u\right) =\frac{1}{2}\int_{\Omega }\left[ 
\frac{1}{\varepsilon }\left( \partial _{z}u-\frac{1}{2}\left( \partial
_{x}u\right) ^{2}\right) ^{2}+\varepsilon \left( \partial _{x}^{2}u\right)
^{2}\right] dx\,dz,  \label{2dsingular}
\end{equation}%
to the 3D energies \eqref{3denergy}. 
For $\varepsilon _{n}\rightarrow 0$ and a sequence $%
\left\{ u_{n}\right\} $ with bounded energies $\mathcal{J}_{\varepsilon
_{n}}\left( u_{n}\right) $, we proved compactness of $\nabla u_{n}$ in $%
L^{q} $ for $1\leq q < p$ under the additional assumption $\| \nabla u_{n}\|_{L^p}\leq C$ {for some $p>6$}. Moreover, we obtained a lower bound on $\mathcal{J}%
_{\varepsilon }$ and constructed a matching upper bound using on a 1D ansatz.  

The sharp lower bound for the 3D energies, which was not previously shown in the physics literature, relies on a calibration argument which briefly works as follows. 
Letting 
\begin{equation}\label{ourent}
\Sigma(\nabla u)=\left(\Sigma_1, \Sigma_2, \Sigma_3\right)
\end{equation}
where 
\begin{eqnarray*}
&&\Sigma_1=\dz u \dx u -\frac{1}{2}\dx u (\dy u)^2-\frac{1}{6}(\dx u)^3,\\
&&\Sigma_2=- \dx u\dy u+\frac{1}{2}\dy u(\dx u)^2+\frac{1}{6}(\dy u)^3,\\
&&\Sigma_3=\frac{1}{2}(\dy u)^2-\frac{1}{2}(\dx u)^2,
\end{eqnarray*}
direct calculation shows that for $u\in H^2(\Omega)$, 
\begin{equation}\label{divesigma}
\dive \Sigma(\nabla u) = \left( \dz u - \frac{1}{2}|\nabla_\bot u|^2 \right) ( \dxx u- \dyy u).
\end{equation}
Thus by the arithmetic mean-geometric mean inequality and the divergence theorem, cf. \eqref{control}, $\mathcal{E}_{\varepsilon }\left( u\right)$ can be bounded below by 
\begin{align}\label{3dsingular-var'}
 \mathcal{E}_{\varepsilon }\left( u\right) &\geq \int_\Omega \dive \Sigma(\nabla u)\ d\ex  - \e\|\nabla_\bot u \|_{H^{\sfrac{1}{2}}(\partial \Omega)}^2 ,
\end{align} 
with approximate equality when  
\begin{equation}\label{appeq}
\frac{1}{\e}{\int_{\Omega}}\left(\dz u - \frac{1}{2}|\nabla_\bot u|^2 \right)^2 \approx \e{\int_{\Omega}}(\Delta_\bot u)^2.
\end{equation}
By the rotational symmetry in the $xy$-plane of the energies \eqref{3denergy}, the same calculation holds for the rotations $\Sigma_{\xi,\eta}$ of $\Sigma$ (see \eqref{rotvar}) obtained by replacing $\{\hat{x},\hat{y}\}$ with another orthonormal basis $\{\xi,\eta \}$ of $\mathbb{R}^2$.
Thus for a sequence $\varepsilon _{n}\rightarrow 0$ and $\left\{
u_{n}\right\} $ converging to a limiting function $u$ in a
suitable space, we may bound $\lim \inf \mathcal{ E}_{\varepsilon_n
}\left( u_n\right) $ from below by taking the supremum of the divergences of $\Sigma_{\xi,\eta}$ over all  $\{\xi,\eta\}$ as in \cite{AmbDeLMan99,AviGig99} for the Aviles-Giga problem. In fact, the energy \eqref{3denergy} contains the 2D Aviles-Giga energy as a special case, which we explain in \Cref{link}.} If $\nabla u\in( BV\cap L^\infty)(\Omega;\mathbb{R}^3)$, this lower bound is given by
\begin{equation}
\int_{J_{\nabla u}} \frac{|\nabla^+_\bot u - \nabla^-_\bot u|^4}{12|\nabla^+ u - \nabla^- u|}\,d\mathcal{H}^2,
\end{equation}
where $J_{\nabla u}$ is the jump set in the sense of $BV$. For the matching upper bound when $\nabla u \in (BV\cap L^\infty)(\Omega;\mathbb{R}^3)$, \eqref{3dsingular-var'} implies the lower bound is optimal for $ \mathcal{E}_{\varepsilon }\left( u\right)$ if \eqref{appeq} holds. When $\Omega$ is a cube, we show this is possible by using a 1D ansatz in which $\nabla u$ varies transverse to the defect set. The ansatz is chosen so that \eqref{appeq} holds. The upper bound for general $\Omega$ can then be shown by appealing to a result of Poliakovsky \cite{Pol08}.


Physically, the 1D ansatz can be interpreted as equating the compression strain $(\partial _{z}u-\frac{1}{2}\left\vert \nabla _{\bot }u\right\vert ^{2})^2$ with the bending strain $\e^2( \Delta _{\bot }u)^2$; see section \Cref{subsec:phys} for physical background of model \eqref{3denergy}. Thus our analysis shows that the frustration coming from the competition between the compression and bending terms is resolved by an equipartition of energy between the two. Moreover, unlike many other problems from materials science where microstructure develops \cite{Koh07}, microstructure does not appear in this smectics model. We remark that the same has been observed in the 2D problem \cite{NY20}. 

\vspace{0.1cm}The compactness of $%
\nabla u_{n}$ in $L^{2}$ is proved under the additional assumptions $%
\| \nabla u_{n}\|_{L^{p}(\Omega)}\leq C $ for some $p>6$ and $\|\nabla u_n \|_{L^2(\partial \Omega)}\leq C$. We emphasize that these are physically justifiable, as the model \eqref{2dsingular} is only valid in the small strain regime \cite{BreMar99,SanKam03}, cf. \Cref{subsec:phys}. Our compactness proof relies on a
compensated compactness argument based on the work of Tartar \cite{Tar79,
Tar83, Tar05} and Murat \cite{Mur78, Mur81ASNP}. The main challenge is to
find the suitable entropies to apply Tartar and Murat's div-curl lemma.
Assuming further that  $\Delta_{\bot}u_n \geq 0$ a.e. in $\Omega$, we show that  $\mathrm{curl}\, E_{n}\text{ and }%
\mathrm{div}\, B_{n}\text{ are compact in }H^{-1}\left( \Omega \right) , $
where 
\begin{equation*}
E_{n}=\left( \nabla _{\bot }u_{n},\frac{1}{2}\left\vert \nabla _{\bot
}u_{n}\right\vert ^{2}\right) \text{ and }\ B_{n}=\left( -\frac{\nabla
_{\bot }u_{n}}{2}\left\vert \nabla _{\bot }u_{n}\right\vert ^{2},\frac{1}{2}%
\left\vert \nabla _{\bot }u_{n}\right\vert ^{2}\right) .
\end{equation*}%
Thus $\left( E_{n},B_{n}\right) $ satisfy the assumptions of the div-curl lemma. Applying the div-curl lemma to $E_{n}\cdot B_{n}$ yields strong convergence of $%
\nabla _{\bot }u_{n}$ in $L^{2}$, and compactness of $\nabla u_{n}$ in $L^{2}$
follows from the fact $ \partial _{z}u_{n}-\frac{1}{2}\left\vert
\nabla _{\bot }u_{n}\right\vert ^{2}\rightarrow 0$ in $L^2$. 

The paper is organized as follows. \Cref{sec:prelim} recalls the physical background of our model \eqref{3denergy}, summarizes the pertinent mathematical literature, and presents the main calculation behind the lower bound in a simplified setting. Also included in \Cref{sec:prelim} are some preliminaries on functions of bounded variation. \Cref{sec:lbd} is devoted to the lower bound. In \Cref{sec:upbd} we construct a sequence which matches the lower
bound from \Cref{sec:lbd} when $\varepsilon \rightarrow 0$, and in \Cref{sec:cpt} we prove compactness. 

\section{Background and Preliminaries} 

\label{sec:prelim}

\subsection{\color{black}Physical background: Smectic A liquid crystals}
\label{subsec:phys}
\indent Smectic liquid crystals are formed by elongated molecules that are aligned and arranged in fluid-like layers. They are a remarkable example of a geometrically frustrated, multi-layer, soft-matter system. Ground states of smectic liquid crystals are characterized by flat, equally spaced, parallel layers. Due to spontaneously broken translational and rotational symmetry, singularities form in regions where the smectic order breaks down. When defects are present, the layers must bend and the resulting curvature is, in general, incompatible with equal spacing between them. The subtle interplay between the geometry of the layers and equal spacing imposes theoretical complications, and understanding the layer structure of a smectic liquid crystal is a challenging task.

Smectics can be represented by the density modulation $\Delta \rho
\varpropto $ $\cos \left[ \frac{2\pi }{a}\phi \left( \mathbf{x}\right)\right] $, where $\mathbf{x}=(x,y,z) \in \mathbb{R}^3$, $a$ is the layer spacing, and $\phi $ is the phase field of the order parameter \cite{CL95, deGP}. The peaks of the density wave where $\phi \left( \mathbf{x}\right) \in a\mathbb{Z}$  correspond to the smectic layers. We are interested in the smectic A phase, in which the nematic director coincides with the normal to the smectic layers $\mathbf{N=}\frac{\nabla \phi }{\left\vert \nabla \phi \right\vert }$. In terms of $\phi $, the free energy of a smectic liquid crystal on a region $\Omega$ \cite{SanKam05} is the sum of the compression and
bending energies 
\begin{equation}
F=\frac{1}{2}\int_{\Omega }B\left[ \left( 1-\left\vert \nabla \phi
\right\vert \right) ^{2}+K_{1}\left( \nabla \cdot \frac{\nabla \phi }{\left\vert \nabla \phi \right\vert }\right) ^{2}\right] d\mathbf{x},
\label{smectics-freeenergy}
\end{equation}
where $B$ is the compression modulus and $K_{1}$ the bend modulus. The constant $\e =\sqrt{\frac{K_{1}}{B}}$ is the penetration length. In the presence of boundaries, there is also the saddle-splay term 
\begin{equation*}
F_{K}=\widetilde{K}\int_{\Omega }\nabla \cdot \left[ \left( \nabla \cdot 
\mathbf{N}\right) \mathbf{N-}\left( \mathbf{N\cdot \nabla }\right) \mathbf{N}%
\right] d\mathbf{x}.
\end{equation*}%
The contribution of this term depends only on the boundary conditions and is
often excluded from the energy. For configurations with topological defects, this term can contribute to the core energy of a defect.

The global minimizer of $\left( \ref{smectics-freeenergy}\right) $ is the zero-energy state $\phi \left( \mathbf{x}\right) =\mathbf{n\cdot x+}\phi _{0}$ with $\mathbf{n }\in \mathbb{S}^{2}$ and $\phi_0 \in \mathbb{R}$. However, rarely do both terms in the free energy vanish. To understand the frustration of the problem in general, the following example of a smooth surface surface $\mathbf{x}_0=\mathbf{x}_0(u,v)$ and its parallel surfaces $\mathbf{x}_{n}(u,v):=\mathbf{x}_{0}(u,v)+na\mathbf{N}(u,v)$, $n \in \mathbb{Z}$ is illuminating \cite{KamSan06}. Let $H_n$ and $K_n$ denote the mean and Gaussian curvatures of $\mathbf{x}_n$. A standard calculation \cite[Sec. 3-5 Exercise 11]{doC76} yields the formulas
\begin{equation*}
H_{n}=\frac{H_{0}-naK_{0}}{1-2naH_{0}+n^{2}a^{2}K_{0}},\text{ }K_{n}=%
\frac{K_{0}}{1-2naH_{0}+n^{2}a^{2}K_{0}}.
\end{equation*}%
Since mean curvature can be expressed in terms of the surface normal by $H=-%
\frac{1}{2}\nabla \cdot \mathbf{N,}$ the bending term is proportional to $%
H^{2}.$ Therefore, for evenly spaced layers, the only way for the bending to
vanish (so $H_{n}=0$) for all layers is for the Gaussian curvature $K_{0}$
of $\mathbf{x}_{0}$ to be zero, which in turn implies $K_{n}=0$, so that the product of the principal curvatures is zero. On the other hand, $H_{n}=0$ implies that the principal curvatures sum to 0 as well, so that they both vanish everywhere. Thus unless all the layers are flat, vanishing curvature is incompatible with the uniform layer spacing. The interaction between the layer spacing, the Gaussian curvature, and the mean curvature presents a major obstacle to finding minimal configurations for the energy $\left( \ref{smectics-freeenergy}\right) $. Throughout the physics literature there are numerous works on the search for exact or approximate solutions of deformations
in smectics \cite{BluKam02, BreMar99, DiK02, DiK03, IL99,
KamLub99, KamSan06, MatKamSan12, San06, SanKam03, SanKam05, SanKam07}.

In the study of smectic layers, it is typical to consider the deviation $u$ from a fixed ground state $\overline{\phi}$, so $\phi =\overline{\phi }-u$. If we fix $\overline{\phi }(\mathbf{x})=z$, then $u(\mathbf{x})=z-\phi (\mathbf{x}).$ Expressing the
compression strain in powers of $\nabla u=\hat{z}-\nabla {\phi }$ and
setting $\nabla _{\bot }$ $=\partial _{x}\hat{x}+\partial _{y}\hat{y}$, $\Delta_\bot = \dxx+\dyy$, we
can expand the compression as 
\begin{equation}
1-\left\vert \nabla \phi \right\vert \approx \partial _{z}u-\frac{1}{2}%
\left\vert \nabla _{\bot }u\right\vert ^{2}+\mathcal{O}\left( u^{3}\right) ,
\label{compexp}
\end{equation}%
and bending strain as%
\begin{equation}
\nabla \cdot \frac{\nabla \phi }{\left\vert \nabla \phi \right\vert }\approx
-\Delta _{\bot }u+\mathcal{O}\left( u^{3}\right) .  \label{bendexp}
\end{equation}%
Keeping only the linear terms in the expansions $\left( \ref{compexp}\right)$ and $\left( \ref{bendexp}\right) $ in the limit of small elastic strains $\left\vert \nabla u\right\vert \ll 1$ results in a linear theory \cite{deGP,Kle83} of $\left( \ref{smectics-freeenergy}\right) $ which has been widely used in the study of strain fields and energetics of dislocations in smectics-A. On the other hand, the well-known example of the dilatative Helfrich-Hurault effect \cite{Hel71, Hur73}, in which the layers wrinkle upon stretching, indicates that nonlinear effects can be important even for small strains. Observe that the linear model is only valid when $\frac{\left\vert \nabla _{\bot }u\right\vert ^{2}}{\partial _{z}u}\ll 1$, and the nonlinear term is not negligible when $\partial _{z}u\sim\left\vert \nabla _{\bot }u\right\vert ^{2}\ $. The first example of strain field taking into account the nonlinear effect was constructed by Brener and Marchenko \cite{BreMar99}. They considered the strain field for a single edge dislocation in the regime $\partial _{z}u\sim \left( \partial _{x}u\right) ^{2}\ll 1$ and found an exact solution to the Euler-Lagrange equation for the 2D nonlinear approximation of $\left( \ref{smectics-freeenergy}\right) $
\begin{equation}
F=\frac{B}{2}\int_{\Omega }\left[ \left( \partial _{z}u-\frac{1}{2}\left(
\partial _{x}u\right) ^{2}\right) ^{2}+\e^{2}\left( \partial
_{x}^2u\right) ^{2}\right] dx\,dz,  \label{2dnonlinear}
\end{equation}%
where 
\begin{equation}\label{e definition}
\e = \sqrt{K_1/B}.
\end{equation}
Their construction deviates markedly from the strain field predicted by
linear theory even away from the defects where the elastic strain and curvature are small. The solution was verified experimentally in a cholesteric finger texture by Ishikawa and Lavrentovich \cite{IL99} and by Smalyukh and Lavrentovich \cite{SmaLav03} using confocal microscopy.

Brener and Marchenko found their solution by solving the fourth order Euler-Lagrange equation of $\left( \ref{2dnonlinear}\right) $ directly. Later, Santangelo and Kamien \cite{SanKam03} approached the problem from a different perspective and discovered a large class of exact minima for the nonlinear approximations of (\ref{smectics-freeenergy}). They studied the 3D version of \eqref{2dnonlinear}
\begin{equation}
F=\frac{B}{2}\int_{\Omega }\left[ \left( \partial _{z}u-\frac{1}{2}%
\left\vert \nabla _{\bot }u\right\vert ^{2}\right) ^{2}+{\e ^{2}}\left(
\Delta _{\bot }u\right) ^{2}\right] d\mathbf{x}.  \label{3dnonlinear}
\end{equation}%
 Completing the square in $\left( \ref{3dnonlinear}\right)$ yields%
\begin{eqnarray}
\mathcal{E}_{\varepsilon }\left( u\right) &=&\frac{1}{2}\int_{\Omega }\left[ 
\frac{1}{\varepsilon }\left( \partial _{z}u-\frac{1}{2}\left\vert \nabla
_{\bot }u\right\vert ^{2}\right) ^{2}+\varepsilon \left( \Delta _{\bot
}u\right) ^{2}\right]\, d\mathbf{x}  \notag \\
&=&\frac{1}{2}\int_{\Omega }\frac{1}{\varepsilon }\left( \partial _{z}u-
\frac{1}{2}\left\vert \nabla _{\bot }u\right\vert ^{2}{\oldmp}\varepsilon \Delta
_{\bot }u\right) ^{2}\,d\mathbf{x}\,{\pm}\int_{\Omega }\left( \partial _{z}u-\frac{1%
}{2}\left\vert \nabla _{\bot }u\right\vert ^{2}\right) \Delta _{\bot }u\,d\mathbf{x}.  \label{ebps}
\end{eqnarray}
Observe that
\begin{equation}
\partial _{z}u\Delta _{\bot }u=\nabla _{\bot }\cdot \left( \partial
_{z}u\nabla _{\bot }u\right) -\frac{1}{2}\partial _{z}\left( {\left\vert
\nabla _{\bot }u\right\vert ^{2}}\right)  \label{cros1}
\end{equation}
and
\begin{equation}
3\left\vert \nabla _{\bot }u\right\vert ^{2}\Delta _{\bot }u=-4\overline{K}%
u+\nabla _{\bot }\cdot \left( \nabla _{\bot }u\left\vert \nabla _{\bot
}u\right\vert ^{2}\right) +2\nabla _{\bot }\cdot \left( u\Delta _{\bot }u\nabla _{\bot
}u-\frac{1}{2}u\nabla _{\bot }\left\vert \nabla _{\bot
}u\right\vert ^{2}\right) ,  \label{cros2}
\end{equation}%
where $\overline{K}=\dxx u \dyy u - (\dxy u)^2$ is the
lowest order approximation of the Gaussian curvature. {\color{black}By the Bochner formula $\frac{1}{2}\Delta |\nabla u|^2=\nabla u\cdot \Delta (\nabla u)+|\nabla^2 u|^2$ from flat geometry or by direct calculation}, $\overline{K}$ can also be expressed as
\begin{equation}
\overline{K}
=\frac{1}{2}\nabla _{\bot }\cdot \left( \nabla _{\bot }u\Delta
_{\bot }u-\frac{1}{2}\nabla _{\bot }\left\vert \nabla _{\bot }u\right\vert
^{2}\right) \label{approxgau},
\end{equation}
which is the form found for example in \cite{SanKam03}. Substituting $\left( %
\ref{cros1}\right) $ and $\left( \ref{cros2}\right) $ back into $\left( \ref%
{ebps}\right) $ yields%
\begin{eqnarray}\label{bpsprelim}
\mathcal{E}_{\varepsilon }\left( u\right) &=&\frac{1}{2}\int_{\Omega }\frac{1%
}{\varepsilon }\left( \partial _{z}u-\frac{1}{2}\left\vert \nabla _{\bot
}u\right\vert ^{2}{\oldmp}\varepsilon \Delta _{\bot }u\right) ^{2}\,d\mathbf{x}\,{\pm\frac{%
2}{3}}\int_{\Omega }\overline{K}u\,d\mathbf{x}  \notag \\
&&{\pm}\int_{\Omega }\mathrm{div}\, \Xi(u) \,d\mathbf{x,}  \label{ebps3d}
\end{eqnarray}%
where 
\begin{equation}\label{calibps}
\Xi (u)=\left[ \left( \partial _{z}u-{\frac{1}{6}}\left\vert \nabla
_{\bot }u\right\vert ^{2}-{\frac{1}{3}}u\Delta _{\bot }u\right) \nabla _{\bot
}u+{\frac{1}{6}}u\nabla _{\bot }\left\vert \nabla _{\bot }u\right\vert ^{2},-%
\frac{1}{2}{\left\vert \nabla _{\bot }u\right\vert ^{2}}\right] .
\end{equation}

A direct conclusion from this decomposition is that the free energy of deformations with $\overline{K}=0$ is always bounded below by the contributions from the boundary integrals involving $\Xi(u)$ , and the minimum is achieved when
\begin{equation}
\partial _{z}u-\frac{1}{2}\left\vert \nabla _{\bot }u\right\vert
^{2}=\pm\e \Delta _{\bot }u.  \label{3dbps}
\end{equation}
The nonlinear differential equation \eqref{3dbps} is called the BPS equation, with solutions referred to as BPS solutions. This equation is of reduced order compared to the Euler-Lagrange equation of the free energy $\left( \ref{3dnonlinear}\right)$. As observed in \cite{SanKam03}, \eqref{3dbps} has a simple but important physical interpretation: equating the bending and compression energies so as to minimize their sum and alleviate the geometric frustration described earlier. This type of technique, called the BPS decomposition,  was introduced by Bogomol'nyi \cite{Bo76}, Prasad and Sommerfield \cite{PraSom76} in the study of field configurations of magnetic monopoles and solitions in field theory. BPS-type decompositions have also been utilized in the analysis of thermal fluctuations in 2D smectics \cite{GolWan94} and shape changes in vesicles \cite{BenSaxLoo01}.

When $u$ is a function of $z$ and $x$ only, so that $\overline{K}=0$, the BPS equation simplifies to
\begin{equation}
\partial _{z}u-\frac{1}{2}\left( \partial _{x}u\right) ^{2}=\pm\e \partial
_{x}^{2}u.  \label{bpsequation}
\end{equation}%
Through the Hopf-Cole transformation $S_{\pm }=\exp[\pm u(x,\pm z)/(2\e) ],$ $%
\left( \ref{bpsequation}\right) $ becomes the diffusion equation%
\begin{equation}
\partial _{z}S_{\pm }=\e \partial _{x}^{2}S_{\pm }.  \label{sequation}
\end{equation}
Solving $\left( \ref{sequation}\right) $ with the boundary conditions $S_{\pm }\rightarrow 1$ as $x\rightarrow -\infty $
and $S_{\pm }\rightarrow e^{\pm \frac{b}{4\e }}$ as $%
x\rightarrow +\infty ,$ where $b$ is the Burgers vector of the dislocation,
yields
\begin{equation*}
S_{\pm }=1+\left( e^{\pm\frac{b}{4\e }}-1\right) \pi ^{-\sfrac{1}{2}}\int_{-\infty
}^{\frac{x}{2\sqrt{\e z}}}e^{-t^{2}}dt.
\end{equation*}%
After inverting the Hopf-Cole transformation, this solution recovers the edge dislocation deformation calculated in \cite{BreMar99}. 

When $\overline{K}$ is small, the BPS solutions are energetically preferable compared to solutions from the linear theory \cite{SanKam03}. Santangelo and Kamien \cite{SanKam05} generalized this idea to the full energy %
\eqref{smectics-freeenergy} and established a specific set of minima of $%
\left( \ref{smectics-freeenergy}\right) $ when the $\overline{K}=0$. Their arguments demonstrated that the layer deformation in the
partially nonlinear theory \cite{BreMar99, SanKam03} is near the
profile from the full energy.  

Given the rigidity of the assumption that the Gaussian curvature of the smectic layers is zero, one might conjecture that BPS solutions are no longer (approximate) minimizers upon relaxing that assumption. Interestingly however, this is not the case. Indeed, we quote from \cite{SanKam03}, where some numerical simulations were done to investigate this issue. ``Further study is needed to understand the precise role of the $u\overline{K}$ in the failure of the BPS configurations to minimize the energy. It is often the case that `near-BPS' solutions are remarkably good approximants and it appears to be true here as well." We utilize techniques drawn from the mathematical literature for singular perturbation problems to obtain a sharp lower bound for \eqref{3dnonlinear} while only assuming that $\int\e\overline{K}$ is small, providing an explanation for this phenomenon. Our analysis is inspired by the simple fact that the 3D smectic energy \eqref{3dnonlinear} is a generalization of the well-studied 2D Aviles-Giga functional, which, nevertheless, has not been observed previously in the literature to the best of our knowledge.\par

\subsection{\color{black}Mathematical Background: The Aviles-Giga energy}\label{link}

To illustrate the link between 3D smectics and Aviles-Giga functional, fix $\Omega \subset \mathbb{R}^2$ and consider the smectic energy \eqref{3dnonlinear} on the three dimensional cylinder $\Omega \times (0,1)$ subject to the constraint $\dz u = \sfrac{1}{2}$. With the dependence of $\nabla u$ on $z$ eliminated, the energy \eqref{3dnonlinear} becomes
\begin{equation}\label{ourversion}
\frac{B}{2}\int_{\Omega} \left[ \frac{1}{4}\left( 1- \left\vert \nabla _{\bot }u\right\vert ^{2}\right) ^{2}+\e ^{2}\left(
\Delta _{\bot }u\right) ^{2}\right] \,dx \,dy.
\end{equation}
Ignoring the harmless factor of $\sfrac{1}{4}$ on the first term, this is the 2D instance of the Aviles-Giga energy, which we now recall.

Aviles and Giga \cite{AviGig87} formulated the energy
\begin{equation}\label{AGoriginal}
\mathcal{F}_{\varepsilon}=\int_{\Omega}\left[\frac{1}{\varepsilon}(|\nabla u|^2-1)^2+\varepsilon(\Delta u)^2\,\right] d\ex,
\end{equation}
known as the Aviles-Giga functional, as a model for smectic liquid crystals. Here $\Omega \subset \mathbb{R}^n$ is a bounded domain.  When $\varepsilon$ approaches zero,  Aviles and Giga conjectured that the optimal transition layers are one dimensional and $\mathcal{F}_{\varepsilon}$ converges (in the sense of $\Gamma$-convergence) to the limiting energy
\begin{equation*}
\mathcal{F}_0=1/3\int_{J_{\nabla u}}|[\nabla u]|^3 \,d\mathcal{H}^{n-1}.
\end{equation*}
Here the limiting function $u$ satisfies eikonal equation $|\nabla u|=1$ a.e., $J_{\nabla u}$ is the defect set, cf. \Cref{jumpdef}, and $[\nabla u]$ is the jump in  $\nabla u$ across $J_{\nabla u}$.\par
The Aviles-Giga functional has been extensively studied in the case $n=2$. After extracting a boundary term,  \eqref{AGoriginal} is equal to 
\begin{equation}\label{Aviles-Giga}
\tilde{\mathcal{F}}_{\varepsilon}=\int_{\Omega}\left[\frac{1}{\varepsilon}(|\nabla u|^2-1)^2+\varepsilon|\nabla^2 u|^2\right]\,dx\,dy.
\end{equation} 
Jin and Kohn \cite{JinKoh00} noticed that the divergence of the ``Jin-Kohn entropy"
\begin{equation}\label{jkentropy}
\left(\dx u\left(1-(\dy u)^2-\frac{1}{3}(\dx u)^2\right),-\dy u\left(1-(\dx u)^2-\frac{1}{3}(\dy u)^2\right)\right),    
\end{equation}
calculated directly as $(1-|\nabla u|^2)(\dxx u-\dyy u)$, bounds $\tilde{\mathcal{F}}_{\varepsilon}$  from below and the lower bound is asymptotically optimal if 
$$
|1-|\nabla u_{\varepsilon}|^2|\approx |\dxx u_{\varepsilon}-\dyy u_{\varepsilon}|\quad \text{and}\quad{\e\int_{\Omega}\left( \dxx u_{\e}\dyy u_{\e}- (\dxy u_{\varepsilon})^2\right)\,dx\,dy\approx 0.}
$$

For the unit square and boundary conditions $u=0$, $\frac{\partial u}{\partial n}=-1$,  they proved that the lower bound can be achieved by the ``1D" ansatz $u_{\varepsilon}=ax+f_{\varepsilon}(y)$ when the associated defect set of the limiting map $\lim_{\varepsilon \rightarrow 0}u_{\varepsilon}$ is parallel to the $x$ axis, corroborating Aviles-Giga's conjecture regarding the one-dimensionality of the transition region. Recently,  Ignat and Monteil \cite{IgnMon20}  proved that any minimizer of \eqref{Aviles-Giga} on an infinite strip is one-dimensional.  By considering the supremum of the divergences of all rotated versions of $\Sigma u$, Aviles and Giga \cite{AviGig99} derived a limiting functional $J: W^{1,3}(\Omega)\rightarrow [0, \infty)$ which is lower semicontinuous with respect to strong topology in $W^{1,3}(\Omega)$ and coincides with $\color{black}\mathcal{F}_0$ for any $u$ satisfying the eikonal equation with $\nabla u\in BV(\Omega)$.  Moreover, 
$$
J(u)\leq \liminf_{n\rightarrow \infty} \tilde{\mathcal{F}}_{\varepsilon_n}(u_n)
$$ 
for any sequence $u_n$ converging strongly to $u$ in $W^{1,3}(\Omega)$. A construction which achieves conjectured $\Gamma$-limit when $\nabla u \in BV(\Omega)$ was provided in \cite{ConDeL07, Pol08}. For the $\Gamma$-convergence theory, another important question is the compactness of sequences with bounded energy when $\varepsilon$ goes to zero.  Such compactness results in two dimensions have been proved by two different groups \cite{AmbDeLMan99, DKMO01} using different approaches.  For the Aviles-Giga functional in dimensions three or higher, the state of art is less clear.  De Lellis \cite{deL02} constructed a counterexample, showing  that $\mathcal{\color{black}F}_0$ is not the limiting energy  for $\mathcal{F}_{\varepsilon}$ and 1D ansatz is not optimal. The compactness and form of the limiting energy, however, are still open. The Aviles-Giga model and related topics such as the eikonal equation and other line-energy models have continued to be active areas of research in the past two decades; see \cite{AloRivSer02,DeLOtt03,GhiLam20, Ign12,Ign12a, IgnMer11,IgnMer12, LamLorPen20, DeLIgn15, Lor12,Lor14,  LorPen18,LorPen21,Mar21,RivSer01,RivSer03} and the references therein.\par

\subsection{Heuristic Proof of the Lower Bound}

The admissible class for \eqref{3denergy} is $H^2(\Omega)$.
We recall some trace properties for this space when $\partial \Omega$ is Lipschitz that are useful for the lower bound. 
First, since $u\in H^2$, it follows that $\left.\nabla u\right|_{\partial \Omega} \in H^{\sfrac{1}{2}}(\partial \Omega;\R^3)$. {\color{black}Furthermore, momentarily replacing $\dx$, $\dy$, and $\dz$ by $\partial_1$, $\partial_2$, and $\partial_3$, the tangential derivative operators
\begin{equation}\label{hminushalf}
\partial_{\tau_{jk}} = \nu_j \partial_k - \nu_k \partial_j : H^{\sfrac{1}{2}}(\partial \Omega) \to H^{-\sfrac{1}{2}}(\partial \Omega)=(H^{\sfrac{1}{2}}(\partial \Omega))^\ast
\end{equation}
are well-defined, linear, and bounded for any $1\leq j,k,l \leq 3$ \cite[Proposition 2.1]{MitMitWri11}.} In particular, denoting by $<,>$ the pairing between $H^{-\sfrac{1}{2}}(\partial \Omega)$ and $H^{\sfrac{1}{2}}(\partial \Omega)$, the quantities
\begin{equation}\label{pairing}
   <\partial_m u, \partial_{\tau_{jk}} (\partial_l u) >
\end{equation}
are well-defined for $u\in H^2(\Omega)$ and $1\leq j,k,l,m \leq 3$.

We recall the map
\begin{align*}
\left(\dz u \dx u -\frac{1}{2}\dx u (\dy u)^2-\frac{1}{6}(\dx u)^3,
- \dx u\dy u+\frac{1}{2}\dy u(\dx u)^2+
\frac{1}{6}(\dy u)^3,\frac{1}{2}(\dy u)^2-\frac{1}{2}(\dx u)^2\right)
\end{align*}
from \eqref{ourent}. We point out that the first two components of $\Sigma(\nabla u)$ are one half times the Jin-Kohn entropy for the Aviles-Giga energy when $\partial_z u=\frac{1}{2}$.  Direct calculation shows that {for $u\in H^2(\Omega)$,}
\begin{align}\notag
\dive \Sigma(\nabla u) &= \left(\dz u - \frac{1}{2}|\nabla_\bot u|^2 \right) (\dx^2 u - \dy^2 u ) \\ \label{dircalc}
&\leq \frac{1}{2\e}\left(\dz u - \frac{1}{2}|\nabla_\bot u|^2 \right) ^2 + \frac{\e}{2} (\dx^2 u - \dy^2 u )^2.
\end{align}
To handle the fact that the second term differs from $(\Delta_\bot u)^2$, we recall $\overline{K}= \dxx u \dyy u - (\dxy u)^2$, so that
\begin{align}\notag
\int_\Omega \left[(\Delta_\bot u)^2 - (\dx^2 u - \dy^2 u)^2 - 4\overline{K} \right] \, d\ex = \int_\Omega 4 (\dxy u)^2\,d\ex \geq 0.
\end{align}
Combining this with \eqref{dircalc}, we arrive at
\begin{align}\label{is}
\int_\Omega \dive \Sigma(\nabla u) \, d\ex + \e\int_\Omega 2\overline{K}\,d\ex \leq \mathcal{E}_\e(u)
.\end{align}
Notice that $\overline{K} = \dive_\bot (\dx u \dyy u, - \dx u \dxy u)$, so that 
\begin{align}\label{control}
\e\int_\Omega 2\overline{K}\, d\ex = 2\e \int_{\partial \Omega} \dx u \partial_{\tau_{12}} \dy u \,d\mathcal{H}^2
\end{align}
and both terms on the left hand side of \eqref{is} depend only on boundary values. 
By the continuity of the tangential derivative operators from $H^{\sfrac{1}{2}}(\partial \Omega)$ to $H^{\sfrac{-1}{2}}(\partial \Omega)$, we can estimate
\begin{align}\label{kcontrol}
\left|\e\int_\Omega 2\overline{K} \,d\ex\right|\leq 2\e \|\dx u \|_{H^{\sfrac{1}{2}}(\partial \Omega)}\|{\color{black}\partial_{\tau_{12}}}\dy u \|_{H^{\sfrac{-1}{2}}(\partial \Omega)}\leq {\color{black}C(\Omega)}\e \| \nabla_\bot u\|_{H^{\sfrac{1}{2}}(\partial\Omega)}^2.
\end{align}
The general lower bound as $\e\to 0$ is then derived by taking supremum over all the rotations of $\Sigma$ as in \eqref{rotvar}.
We also refer the reader to the end of \Cref{sec:upbd} for a discussion of the implications of this analysis on the BPS solutions.\par
{\color{black}\begin{remark}
There are few 3D examples where an explicit calibration can be found and, to our knowledge, there is no systematic approach to find such a calibration for an arbitrary energy. Our choice of the calibration \eqref{ourent} uses only  $\nabla u$ while the calibration \eqref{calibps}  from BPS decomposition involves second derivatives of $u$.   For $u$ smooth, since 
\begin{eqnarray*}
&&(\Delta_{\bot}u)^2-(\dxx u-\dyy u)^2=4\overline{K},\\
&& (\dx u-\frac{1}{2}|\nabla_{\bot} u|^2)\cdot \Delta_{\bot}u=\dive \Xi(u)+\frac{2}{3}\overline{K}u,\\
&&(\dx u-\frac{1}{2}|\nabla_{\bot} u|^2)\cdot (\dxx u-\dyy u)=\dive\Sigma(\nabla u),
\end{eqnarray*}
the two calibrations can be linked by the following identity when $\overline{K}=0$:
\begin{equation*}
|\dive \Xi|=|\dive \Sigma(\nabla u)|.
\end{equation*}

\end{remark}}
\begin{remark}
Comparing the BPS decomposition to \eqref{is}, we see that the remainder term $\e\int_\Omega 2\overline{K}$ in the latter can be handled more easily than the term $\int_\Omega u \overline{K}$. This is the reason that we are able to obtain a lower bound even in the presence of non-vanishing Gaussian curvature. 
One might guess that since the BPS decomposition is also predicated on equipartition of energy between the bending and compression terms, the two arguments give the same lower bound 
\begin{equation}\notag
\int_{J_{\nabla u}} \frac{|\nabla^+_\bot u - \nabla^-_\bot u|^4}{12|\nabla^+ u - \nabla^- u|}\,d\mathcal{H}^2
\end{equation}
in the limit $\e \to 0$ when $\overline{K}=0$ for each $u_\e$ and $\nabla u \in BV(\Omega;\mathbb{R}^3)$. This is indeed the case, although passing to the limit as $u_\e \to u$ in the term $\int_\Omega \dive \Xi(u_\e)$ is non-trivial since $\Xi (u_\e)$ contains second order derivatives of $u_\e$. This can be accomplished via a blowup argument which then allows for a careful analysis of those higher order terms in the simplified setting of a flat jump set with limiting constant states $\nabla u^\pm$ on either side. 
\end{remark}

}

\par

\subsection{Properties of Functions of Bounded Variation} Our discussion draws from the relevant sections of \cite[Chapter 3]{AmbFusPal00}. For the sake of generality and because the dimensions of the ambient/target spaces do not matter for these results, in this subsection we will consider functions defined on $\Omega\subset \mathbb{R}^N$ and taking values in $\mathbb{R}^M$.  
\begin{definition}{\cite[Def. 3.1]{AmbFusPal00}}
An element $m \subset L^1(\Omega; \mathbb{R}^M)$ belongs to the space $BV(\Omega;\R^M)$ if the distributional derivative $Dm=\left(D_j m_i \right)$ of $m$ is a finite $\R^{M\times N}$-valued Radon measure.
\end{definition}
\begin{definition}{\cite[Def. 3.63]{AmbFusPal00}}
If $m \in L^1(\Omega;\R^M)$ we say that $m$ has approximate limit $z=\ap \lim_{y \to x} m(y)$ at $x \in \Omega$ if
\begin{equation}\notag
\lim_{r \to 0} \fint_{B_r(x)} | m(y) - z |\,dy =0.
\end{equation}
If this property fails to hold at $x$ for every $z\in \mathbb{R}^M$, then $x$ belongs to $S_m$, the approximate discontinuity set.
\end{definition}
To refer to solid half-balls in $\Omega$, we define
\begin{equation}\notag
    B_r^+(x,\nu) := \left\{y \in B_r(x) : (y-x)\cdot \nu > 0 \right\},B_r^-(x,\nu) := \left\{y \in B_r(x) : (y-x)\cdot \nu < 0 \right\}.
\end{equation}
\begin{definition}{\cite[Def. 3.67]{AmbFusPal00}}\label{jumpdef}
For $m\in BV(\Omega;\R^M)$, we say that $x\in \Omega$ belongs to $J_m$, the set of approximate jump points, if there exist $m^+(x)\neq m^-(x)\in \R^M$ and $\nu_m(x) \in \mathbb{S}^{N-1}$ such that
\begin{equation*}
    \lim_{r \to 0} \fint_{B_r^+(x,\nu_m)} |m(y) - m^+(x) |\, dy=0\textit{ and } \lim_{r \to 0} \fint_{B_r^-(x,\nu_m)} |m(y) - m^-(x) |\, dy=0.
\end{equation*}
The vectors $m^+, m^-,$ and $\nu_m$ are uniquely determined up to permuting $m^+, m^-$ and exchanging $\nu_m(x)$ for $-\nu_m(x)$. Also, $J_m$ is countably $\mathcal{H}^{N-1}$-rectifiable. 
\end{definition}
\begin{definition}{\cite[Cor. 3.80]{AmbFusPal00}}
The precise representative  of $m\in BV(\Omega;\R^M)$ is the function
\begin{equation}\notag
    \tilde{m}(x) :=  \begin{cases} 
      \ap \lim_{y \to x}m(y) & \textit{ if }x\notin S_m, \\
      \displaystyle\frac{m^+(x) + m^-(x)}{2} & \textit{ if }x\in J_m .
   \end{cases}
\end{equation}
\end{definition}
\begin{theorem}{\cite[Thm. 3.78]{AmbFusPal00}}\label{fedvolp}
If $m\in BV(\Omega;\R^M)$, then $S_m$ is countably $\hn$-rectifiable and $\hn(S_m \setminus J_m)=0$.
\end{theorem}
Next, we recall the BV Structure Theorem.
\begin{theorem}{\cite[Section 3.9]{AmbFusPal00}}
For $m\in BV(\Omega;\R^M)$, the Radon measure $Dm$ can be decomposed into three mutually singular measures
\begin{equation}\label{threesing}
    Dm = D^am + D^j m + D^c m.
\end{equation}
The first component is absolutely continuous with respect to the Lebesgue measure and is given by 
\begin{equation*}
D^a m = \nabla m \mathcal{L}^N,
\end{equation*}
where $\nabla m$ is the matrix of approximate partial derivatives defined $\mathcal{L}^n$-a.e. The component of $Dm$ that is singular with respect to $\mathcal{L}^N$ is $D^sm$. It can be written as $D^s m = D^jm + D^cm$, where
\begin{equation}\label{singpart}
    D^j m = (m^+ - m^-) \otimes \nu_m \hn \mres J_m
\end{equation}
and $D^c m$ is the Cantor part of $Dm$, which vanishes on sets that are $\mathcal{H}^{N-1}$ $\sigma$-finite.
\end{theorem}
\begin{remark}\label{applimrem} Since $D^a m$ and $D^c m$ both vanish on sets that are $\mathcal{H}^{N-1}$ $\sigma$-finite, they vanish on $S_m$. Therefore, $\ap \lim_{y \to x}m(y)$, which is defined off of $S_m$, exists and is equal to $\tilde{m}$ except on a set of $|D^am|$- and $|D^cm|$-measure zero.
\end{remark}
\begin{lemma}
If $m\in BV(\Omega;\mathbb{R}^N)$ is equal to $\nabla u$ for some $u\in W^{1,1}(\Omega)$, then
\begin{equation}\label{jumpcond}
    \nabla u^+ - \nabla u^- \parallelsum \nu_{\nabla u}.
\end{equation}
\end{lemma}
\begin{proof}
Since the Radon measure $Dm$ is equal to $\nabla^2 u$, it is symmetric and can be decomposed into three mutually singular measures, cf. \eqref{threesing}. Thus the jump part $(\nabla u^+ - \nabla u^-)\otimes \nu_{\nabla u}\mathcal{H}^2\mres J_{\nabla u} $ is symmetric as well. But $(\nabla u^+ - \nabla u^-)\otimes \nu_{\nabla u}$ is symmetric if and only if $\nabla u^+ - \nabla u^-\parallelsum  \nu_{\nabla u} $, which is \eqref{jumpcond}.
\end{proof}
We state the BV Chain Rule \cite{Vol67,AmbDal90}.
\begin{theorem}\label{bvchain}{\cite[Thm 3.96]{AmbFusPal00}}
Let $m \in BV(\Omega;\R^M)$ and $F:\R^M \to \R^P$ be $C^1$ with bounded gradient and $F(0)=0$ if $\mathcal{L}^N(\Omega)=\infty$. Then $F\circ m\in BV(\Omega;\R^P)$ and
\begin{equation}\label{chainrule}
    D(F \circ m) = \nabla F(m) \nabla m \mathcal{L}^N + \nabla F(\tilde{m}) D^c m + \left(F(m^+) - F(m^-) \right) \otimes \nu_m \hn \mres J_m.
\end{equation}
\end{theorem}
When $M=P$, taking the trace on both sides of \eqref{chainrule} yields
\begin{equation}\label{chaindiv}
    \dive (F\circ m) = \tr (\nabla F(m) \nabla m) \mathcal{L}^2 + \tr(\nabla F(\tilde{m}) D^c m) + (F(m^+) - F(m^-)) \cdot \nu_m \hn \mres J_m. 
\end{equation}

\begin{remark}\label{bv bound chain}\color{black}
As a consequence of \Cref{bvchain}, if $F$ does not have bounded gradient, one must assume that  $m$ instead is bounded in order to apply the chain rule above.
\end{remark}

\section{The 3D Aviles-Giga Space and the Lower Bound} 

\label {sec:lbd}
In this section we prove the lower bound. Due to the connection between the 3D model \eqref{3dnonlinear} and 2D Aviles-Giga, many of the arguments leading to the lower bound are the natural 3D analogues of results from \cite{AmbDeLMan99,AviGig99}. The proof is centered around the 3D version of the 2D Aviles-Giga space considered in \cite{AviGig99} and explicitly defined in \cite{AmbDeLMan99}. Although alternate proofs are available when $\nabla u \in BV(\Omega;\mathbb{R}^3)$, for example via blowup or covering arguments, we follow the structure of \cite{AmbDeLMan99} which gives the most general version of the lower bound without this assumption on $\nabla u$.\par
Let $\{\xi,\eta\}$ be an orthonormal basis of $\mathbb{R}^2$. For any vector $m=(m_{\bot},m_3)$ with $m_{\bot}\in \mathbb{R}^2$, denote
\begin{equation}\notag
    m_\xi = m_{\bot} \cdot \xi,\quad m_\eta = m_{\bot}\cdot \eta
\end{equation}
and set
\begin{align}\label{rotvar}
    \Sigma_{\xi\eta}(m) = \left( m_3 m_\xi - \frac{m_\xi m_\eta^2}{2} - \frac{m^3_\xi}{6}\right) \xi   + \left( -m_3 m_\eta + \frac{m_\eta m_\xi^2}{2} + \frac{m^3_\eta}{6}\right) \eta+ \left(-\frac{m_\xi^2}{2}+\frac{m_\eta^2}{2} \right){\color{black}\hat z}.
\end{align}
Note that for $m:\Omega \to \mathbb{R}^3$, $\Sigma_{\xi\eta}(m)\in L^1_{\mathrm{loc}}(\Omega;\mathbb{R}^3)$ if $m_\bot \in L^3_{\mathrm{loc}}(\Omega)$ and $m_3\in L^{\frac{3}{2}}_{\mathrm{loc}}(\Omega)$. This and the fundamental equation \eqref{is} motivate the following definition.
\begin{definition}
Let $u\in W^{{\color{black} 1,}\frac{3}{2}}_{\mathrm{loc}}(\Omega)$ be such that $\nabla_\bot u \in L^3_{\mathrm{loc}}(\Omega;\mathbb{R}^2)$. We say that $u\in AG^{3D}(\Omega)$ if $\dive \Sig(\nabla u)$ is a finite Radon measure in $\Omega$ for all orthonormal bases $\{\xi,\eta\}$ of $\R^2$.
\end{definition}
\begin{definition}
If $u\in \ag$ and $\dz u = \frac{1}{2}|\nabla_\bot u|^2$, we say $u\in \ago$.
\end{definition}
Let $\{e_1,e_2\}$ be the standard basis of $\R^2$, and let $\{\e_1,\e_2\}$ be given by
\begin{equation}
    \e_1 = \left(\frac{1}{\sqrt{2}},\frac{1}{\sqrt{2}}  \right),\quad \e_2 = \left(-\frac{1}{\sqrt{2}},\frac{1}{\sqrt{2}} \right).
\end{equation}
A routine calculation, yields a formula for $\Sig$ in terms of $\Sigma_{e_1e_2}$ and $\Sigma_{\e_1\e_2}$. The $\R^2$-valued version of this formula for the Jin-Kohn entropies was first derived in \cite{AmbDeLMan99}, and we do not include the proof. 
\begin{lemma}
For any $m \in L^{\frac{3}{2}}_{\mathrm{loc}}(\Omega;\mathbb{R}^3)$ such that $m_\bot \in L^3_{\mathrm{loc}}(\Omega;\mathbb{R}^2)$, if
$\xi = (\cos \theta,\sin\theta)$ and $\eta  = (-\sin \theta,\cos \theta)$, then
\begin{equation}\label{entropycombo}
    \Sigma_{\xi\eta}(m) = \cos 2 \theta\, \Sigma_{e_1e_2}(m) + \sin 2\theta\, \Sigma_{\e_1\e_2}(m).
\end{equation}
\end{lemma}
\begin{definition}
For any $u \in \ag$, let $Iu$ be the finite vector-valued Radon measure
\begin{equation}
    I u = (\dive \Sigma_{e_1e_2}(\nabla u) , \dive \Sigma_{\e_1\e_2}(\nabla u)).
\end{equation}
\end{definition}
By \eqref{entropycombo}, $u \in \ag$ if and only if $\dive \Sigma_{e_1e_2}(\nabla u)$, $\dive \Sigma_{\e_1\e_2}(\nabla u)$ are finite Radon measures. {\color{black}Let $\mathcal{B}_{+} $ denote the sets of all orthonormal bases of $\mathbb{R}^2$ with the same orientation as $\{e_1,e_2\}$. We decompose the domain and on each piece we consider the rotated measure $\dive\Sigma_{\xi\eta}(\nabla u) $. }The following result shows that $|Iu|$ is the supremum of the measures $|\dive \Sig(\nabla u)|$ {\color{black} over $\{\xi,\eta\}\in \mathcal{B}_+$}.
\begin{theorem}\label{repformula} 
\begin{enumerate}[(i)]
\item For any $u\in \ag$ and Borel subset $B$ of $\Omega$,
\begin{align}\notag
  \hspace{-.5cm} |Iu|(B) &= \left(\bigvee_{\{\xi,\eta \}{\color{black}\in \mathcal{B}_+}} \left| \dive \Sigma_{\xi\eta}(\nabla u)\right| \right)(B)\\ \label{regform}
    &= \sup \left\{\sum_{j=1}^J \left|\dive\Sigma_{\xi_j\eta_j}(\nabla u) \right|(B_j): J \in \mathbb{N}, \textup{ $\{B_j\}_{j=1}^J$ is a Borel partition of }B, {\color{black}\{\xi_j,\eta_j \} \in \mathcal{B}_{+} }\right\}
\end{align}
\item For an open set $A$, 
\begin{equation}\label{openformula}
   \hspace{-.5cm} |Iu|(A)= \sup \left\{\sum_{j=1}^J \left|\dive\Sigma_{\xi_j\eta_j}(\nabla u) \right|(A_j):J \in \mathbb{N}, A_j \textup{ are open, disjoint, }A_j\ssubset A, {\color{black}\{\xi_j,\eta_j \} \in \mathcal{B}_{+} }\right\}.
\end{equation}
\item If $\{u_n \}\subset \ag$ are such that
\begin{equation*}
    \nabla_\bot u_n \underset{L^3}{\to} \nabla_\bot u,\quad \dz u_n \underset{L^{3/2}}{\to}\dz u,
\end{equation*}
and
\begin{equation}\label{measbound}
    |Iu_n| (\Omega) \leq C < \infty \quad \textup{ for all }n,
\end{equation}
then $u\in \ag$ and
\begin{equation}\notag
    I u_n \overset{\ast}{\rightharpoonup} Iu.
\end{equation}
Thus for every open set $A\subset\Omega$,
\begin{equation}\label{lowers}
    |Iu|(A) \leq \liminf_{n\to \infty} |Iu_n|(A).
\end{equation}
\end{enumerate}
\end{theorem}
\begin{proof}
By Riesz's Theorem and \eqref{entropycombo}, we have the equality of measures
\begin{equation*}
    \dive \Sig (\nabla u)  = (\cos 2\theta, \sin 2\theta) \cdot Iu = (\cos 2\theta,\sin 2\theta) \cdot g |Iu|
\end{equation*}
for some $g$ which is unit-valued $|Iu|$-a.e. One then has for any Borel set $B$ and Borel partition $\{ B_j\}_{j=1}^J$
\begin{equation}
    \sum_{j=1}^J \left| \dive \Sigma_{\xi_j \eta_j} (\nabla u)\right|(B_j) = \sum_{j=1}^J \int_{B_j}\left|(\cos 2\theta_j,\sin 2\theta_j) \cdot g \right| \,d|Iu|\leq \sum_{j=1}^J |Iu|(B_j) .
\end{equation}
Approximating $g$ by functions {\color{black} that take finitely many values in $\mathbb{S}^1$} yields \eqref{regform}. The representation \eqref{openformula} when $A$ is open is a consequence of the usual approximation theorems for Radon measures. To see this, note that for any partition $\{B_j\}_{j=1}^J$ of $A$, we can approximate $B_j$ from inside by disjoint compact sets $\{K_j\}_{j=1}^J$ and then $K_j$ by the desired open sets. \par
For $(iii)$, by H\"older's inequality and the convergence of $\nabla u_n $ to $\nabla u$, we have
\begin{equation}\notag
    \Sigma_{e_1e_2}(\nabla u_n) \to \Sigma_{e_1e_2}(\nabla u ),\quad \Sigma_{\e_1\e_2}(\nabla u_n)\to \Sigma_{\e_1\e_2}(\nabla u)\quad\textup{in }L^1(\Omega;\mathbb{R}^3),
\end{equation}
so that the divergences converge in the sense of distributions. Thus \eqref{measbound} implies that $Iu$ is Radon and $Iu_n \overset{\ast}{\rightharpoonup} Iu$, and \eqref{lowers} is a consequence of the weak-$\ast$ convergence.
\end{proof}
\begin{proposition}
If $u\in W^{2,\sfrac{9}{5}}(\Omega)$, then
\begin{equation}\label{admcalc}
\dive \Sigma_{\xi \eta} (\nabla u) = \left(\dz u - \frac{1}{2}|\nabla u|^2 \right)(\dZZ u - \dEE u)
\end{equation}
and
\begin{equation}\label{h2form}
|Iu | = \left|\dz u - \frac{1}{2}|\nabla_\bot u|^2 \right| | \lambda_1 - \lambda_2|\mathcal{L}^3\mres \Omega   
\end{equation}
where $\lambda_1$, $\lambda_2$ are the eigenvalues of $\nabla^2_\bot u$. 
\end{proposition}
\begin{proof}
If $u$ is smooth, we can calculate
\begin{align} \notag
    \dive \Sigma_{\xi\eta}( \nabla u) 
    &= \dZ \left(\Sigma_{\xi\eta}(\nabla u) \cdot \xi \right) + \dE \left( \Sigma_{\xi\eta}(\nabla u) \cdot \eta \right) + \dz \left(\Sigma_{\xi\eta}(\nabla u) \cdot \hat z \right)\\ \notag
    &= \dZz u \dZ u + \dz u \dZZ u - \dZ u \dE u \dZE u - \frac{\dZZ u (\dE u)^2}{2}- \frac{(\dZ u)^2 \dZZ u}{2} \\ \notag
    &\quad - \dEz u \dE u - \dz u \dEE u + \dZ u \dE u \dEZ u + \frac{\dEE u (\dZ u)^2}{2}+ \frac{(\dE u)^2\dEE u}{2}\\ \notag
    &\quad - \dZ u \dzZ u + \dE u \dzE u \\ \label{jkcalc}
    &= \left(\dz u - \frac{|\nabla_\bot u|^2}{2} \right)\left(\dZZ u - \dEE u\right).
\end{align}
If $u\in W^{2,\sfrac{9}{5}}(\Omega)$, then by the Sobolev embedding, $\nabla u\in L^{\sfrac{9}{2}}(\Omega)$. For $u_n$ smooth and converging to $u$ in ${\color{black}W^{2,\sfrac{9}{5}}(\Omega)}$ , H\"older's inequality yields {\color{black} for any test function $\varphi\in C_0^{\infty}(\Omega),$}
\begin{align*}
\int_\Omega -\Sigma_{\xi\eta}(\nabla u) \cdot \nabla \varphi\,d\ex &=\lim_{n \to \infty}
\int_\Omega  -\Sigma_{\xi\eta}(\nabla u_n) \cdot \nabla \varphi\,d\ex \\
&=\lim_{n \to \infty}\int_\Omega  \left(\dz u_n - \frac{|\nabla_\bot u_n|^2}{2} \right)\left(\dZZ u_n - \dEE u_n\right) \varphi \,d\ex \\
& =\int_\Omega  \left(\dz u - \frac{|\nabla_\bot u|^2}{2} \right)\left(\dZZ u - \dEE u\right)\varphi\,d\ex,
\end{align*}
so that \eqref{admcalc} is proved.
\par
For \eqref{h2form}, if $v=v(\ex)$ and $w=w(\ex)$ are the orthonormal eigenvectors of $\nabla_\bot^2 u(\ex)$ with corresponding eigenvalues $\lambda_1(\ex)=\partial_v^2 u(\ex)$ and $\lambda_2(\ex)=\partial_w^2 u(\ex)$, then direct calculation gives
\begin{align}\notag
\left|\dZZ u - \dEE u \right| = \left|\partial_{v}^2 u - \partial_{w}^2 u \right| \left|(\xi \cdot v)^2 - (\eta \cdot v)^2 \right| =  \left|\partial_{v}^2 u - \partial_{w}^2 u \right| \left|(\xi \cdot w)^2 - (\eta \cdot w)^2 \right| .
\end{align}
From this we may conclude that
\begin{equation*}
\sup_{\{\xi,\eta\} {\color{black}\in \mathcal{B}_+}} \left|\dZZ u - \dEE u \right| = |\lambda_1 - \lambda_2|,
\end{equation*}
and thus
\begin{equation*}
|Iu|=\left|\dz u - \frac{1}{2}|\nabla_\bot u|^2 \right| | \lambda_1 - \lambda_2|\mathcal{L}^3\mres \Omega.
\end{equation*}\end{proof}
{\color{black}
\begin{remark}
 The condition that $u \in W^{2,\frac{9}{5}}(\Omega)$ is stronger than merely requiring that $u \in AG^{3D}(\Omega)$. However, by H\"older's inequality and the Sobolev embedding, $9/5$ is the optimal exponent for which the measure $\dive \Sigma(\nabla u)$is  absolutely continuous with respect to the Lebesgue measure and is thus represented by the integration of an $L^1$-function; cf. \cite[Proposition 3.4]{AmbDeLMan99} for the corresponding result in two dimensions.
\end{remark}
}

The next proposition gives a formula for $|Iu|$ under certain regularity conditions on $u\in \ago$ necessary to apply the BV Chain rule, cf. \Cref{bvchain} {\color{black}and \Cref{bv bound chain}.}
\begin{proposition}\label{bvform}
If $u\in \ago \cap W^{1,\infty}(\Omega)$ and $\nabla u \in BV(\Omega;\mathbb{R}^3)$, then
\begin{equation*}
    |Iu| = \frac{|\nabla_\bot u^+ - \nabla_\bot u^-|^4}{12|\nabla u^+ - \nabla u^-|}{\color{black}\mathcal{H}}^2\mres J_{\nabla u}.
\end{equation*}
\end{proposition}
\begin{proof}
For any $\{\xi,\eta\}$, first notice that due to the BV Chain rule,
\begin{equation*}
    |\dive \Sig (\nabla u)| \mres J_{\nabla u} = \left|\left( \Sig(\nabla u^+) - \Sig(\nabla u^-)\right)\cdot \nu_{\nabla u} \right|\mathcal{H}^2\mres J_{\nabla u}.
\end{equation*}
We compute the right hand side and then optimize over choices of $\{\xi,\eta\}{\color{black}\in \mathcal{B_{+}}}$ at each point in $J_{\color{black}{\nabla u}}$. To simplify the notation in the calculation, set $\nabla u=m$. Now since $u\in \ago$, $m^+$ and $m^-$ satisfy
\begin{equation}\label{use}
m_3^\pm  = \frac{1}{2} ((m_\xi^\pm) ^2 + (m_\eta^\pm) ^2)
\end{equation}
on $J_{m}$, which can be directly verified from \Cref{jumpdef}. This gives
\begin{align}\notag
\Sig(m^\pm)_\bot & = \left( m_3^\pm m_\xi^\pm  - \frac{m_\xi^\pm  (m_\eta^\pm )^2}{2} - \frac{(m_\xi^\pm )^3}{6}\right) \xi   + \left( -m_3^\pm m_\eta^\pm  + \frac{m_\eta^\pm  (m_\xi^\pm )^2}{2} + \frac{(m_\eta^\pm )^3}{6}\right) \eta \\ \label{onthejump}
&= \frac{(m_\xi^\pm )^3}{3} \xi - \frac{(m_\eta^\pm )^3}{3} \eta.
\end{align}
Using \eqref{onthejump} to rewrite $\Sig(m^\pm)$ and then \eqref{use} and $m^+ - m^-\parallelsum \nu_{m}$ to replace $\nu_{ m}$, cf. \Cref{jumpcond}, we have
\begin{align}\notag
    |&( \Sig(m^+) - \Sig(m^-))\cdot \nu_{m} | \\ \notag
    &=  \left|\left(\frac{(m_\xi^+)^3}{3}-\frac{(m_\xi^-)^3}{3} \right)\nu_\xi - \left(\frac{(m_\eta^+)^3}{3}-\frac{(m_\eta^-)^3}{3} \right)\nu_\eta + \left(-\frac{(m_\xi^+)^2}{2}+\frac{(m_\xi^-)^2}{2}+\frac{(m_\eta^+)^2}{2}-\frac{(m_\eta^-)^2}{2} \right)\nu_{\color{black}z} \right| \\ \notag
    &= \frac{1}{|m^+ - m^-|}\Bigg| \left(\frac{(m_\xi^+)^3}{3}-\frac{(m_\xi^-)^3}{3} \right)(m_\xi^+ - m_\xi^-) - \left(\frac{(m_\eta^+)^3}{3}-\frac{(m_\eta^-)^3}{3} \right)(m_\eta^+ - m_\eta^-)\\ \notag
    &\qquad\qquad+ \left(-\frac{(m_\xi^+)^2}{2}+\frac{(m_\xi^-)^2}{2}+\frac{(m_\eta^+)^2}{2}-\frac{(m_\eta^-)^2}{2} \right)\left(\frac{|m_\bot^+|^2}{2}-\frac{|m_\bot^-|^2}{2} \right) \Bigg|.
\end{align}
Expanding out the right hand side of the previous equation and combining like terms gives 
\begin{align}\label{jumpcost}
    |( \Sig(m^+) - \Sig(m^-))\cdot \nu_{m} | = \frac{\left|(m_\xi^+ - m_\xi^-)^4 - (m_\eta^+ - m_\eta^-)^4 \right|}{12|m^+ - m^-|} \leq \frac{|m_\bot^+ - m_\bot^-|^4}{12|m^+ - m^-|}.
\end{align}
Equality is achieved for $\{\xi,\eta\}$ such that $(m^+-m^-)_\bot \parallelsum \xi$ or $(m^+-m^-)_\bot \parallelsum \eta$. 
Taking the supremum over $\{\xi,\eta\}{\color{black}\in \mathcal{B}_{+}}$, we find that
\begin{equation}\label{jumpdens}
    |Iu|\mres J_{\nabla u} = \frac{|\nabla u_\bot^+ - \nabla u_\bot^-|^4}{12|\nabla u^+ - \nabla u^-|}\mathcal{H}^2\mres J_{\nabla u}.
\end{equation}
To complete the proof of the proposition, we must show that
\begin{equation*}
    |Iu|(\Omega \setminus J_{\nabla u})=0.
\end{equation*}
Recalling the BV Chain rule, \Cref{bvchain}, notice that away from $J_{\color{black}\nabla u}$, $\dive \Sig(\nabla u)$ can be computed using the usual chain rule formula by substituting $D^a (\nabla u)$ and $D^c(\nabla u)$ for the classical second derivatives of $u$. Therefore, by the same manipulations as in \eqref{jkcalc}, we have
{\begin{align*}
    \dive \Sig(\nabla u)&\mres (\Omega \setminus J_{\nabla u}) \\ 
&= \left(\widetilde{\dz u}  - \frac{|\widetilde{\nabla_\bot u}|^2}{2}\right)\big[  \xi^T D^a(\nabla u) \xi  +  \xi^T D^c(\nabla u)\xi- \eta^T D^a(\nabla u) \eta  -  \eta^T D^c(\nabla u)\eta \big].
\end{align*}}
Since $\widetilde{\dz u}  - \frac{|\widetilde{\nabla_\bot u}|^2}{2}=0$ for $\ex$ where the approximate limit $\widetilde{\nabla u}$ exists, it is zero $|D^a(\nabla u)|$- and $|D^c(\nabla u)|$-a.e. by \Cref{applimrem}. Thus $|Iu|\mres(\Omega \setminus J_{\nabla u})$ vanishes as well.
\end{proof}
We are ready to prove the lower bound. The theorem is stated under the assumption that $\e_n^2\int \overline{K}_n \to 0$, which can be enforced by mild control on the boundary data as in \eqref{kcontrol}.
\begin{theorem}\label{3dlbd}
Let $\Omega \subset \mathbb{R}^3$ be an open set. Consider $\e_n \searrow 0$ and ${\{ u_n \}\subset H^2(\Omega)}$ such that
\begin{equation}\label{l2conv}
     \nabla_\bot u_n \underset{L^3}{\to} \nabla_\bot u,\quad \dz u_n \underset{L^{3/2}}{\to}\dz u
\end{equation}
for some $u \in W^{1,\frac{3}{2}}(\Omega)$ with $\nabla_\bot u \in L^3(\Omega;\mathbb{R}^2)$. If {$\liminf_{n\to \infty}\mathcal{E}_{\e_n}(u_n)$ is finite and
\begin{equation}\label{vancurv}
\lim_{n\to \infty} \e_n^2\int_{\Omega}\overline {K}_n\,d\ex=0 ,
\end{equation}}
then $u\in\ago$ and
\begin{equation}\label{liminflapl}
   \liminf_{n\to \infty} \mathcal{E}_{\e_n}(u_n) \geq |Iu|(\Omega).
\end{equation}
When $u\in \ago \cap W^{1,\infty}(\Omega)$ and $\nabla u \in BV(\Omega;\mathbb{R}^3)$, then by \Cref{bvform}, the lower bound is given by
\begin{equation*}
    |Iu| = \frac{|\nabla_\bot u^+ - \nabla_\bot u^-|^4}{12|\nabla u^+ - \nabla u^-|}{\color{black}\mathcal{H}}^2\mres J_{\nabla u}.
\end{equation*}
\end{theorem}
\begin{proof}
{For each {\color{black}$u_n\in H^{2}(\Omega)$}, the representation \eqref{h2form} of $|I u_n|$ gives
\begin{align}
    \notag
|I u_n | &= \left|\dz u_n - \frac{1}{2}|\nabla_\bot u_n|^2 \right| | \lambda_{1,n} - \lambda_{2,n}|\mathcal{L}^3\mres \Omega\\ \label{jkcalc2}
&\leq \left[\frac{1}{2\e_n}\left(\dz u_n - \frac{|\nabla_\bot u_n|^2}{2} \right)^2 + \frac{\e_n}{2}\left(\lambda_{1,n} - \lambda_{2,n} \right)^2\right]\mathcal{L}^3\mres \Omega.
\end{align}
Next, fix an open set $A \ssubset \Omega$ and a test function $\varphi \in C^\infty_c(\Omega;[0,1])$ such that $\varphi=1$ on $A$. We estimate
\begin{align}\notag
    \frac{\e_n}{2} \int_{\Omega}(\Delta_\bot u_n)^2  \,d\mathbf{x}&\geq\frac{\e_{\color{black}n}}{2} \int_{\Omega} \left[\lambda_{1,n}^2 +\lambda_{2,n}^2 +2\lambda_{1,n}\lambda_{2,n}\right] \varphi \,d\mathbf{x}\\ \notag
&= \frac{\e_n}{2} \int_{\Omega} \left[(\lambda_{1,n}-\lambda_{2,n})^2 + 4\det(\nabla^2_\bot u_n)\right] \varphi \,d\mathbf{x}\\  \notag
&=\frac{\e_n}{2} \int_{\Omega} (\lambda_{1,n}-\lambda_{2,n})^2  \varphi \,d\mathbf{x}-2\e_n\int_\Omega \dx u_n( \dyy u_n \dx \varphi - \dxy u_n \dy \varphi)\,d\mathbf{x}\\ \label{nulllag}
&\geq \frac{\e_n}{2} \int_{A} \left(\lambda_{1,n} - \lambda_{2,n} \right)^2   \,d\mathbf{x}-2\e_n\|\nabla\varphi \|_{L^\infty}\|{\nabla_\bot} u_n \|_{L^2}{\|\nabla_\bot^2 u_n \|_{L^2}}.
\end{align}
{Now since $\nabla_\bot u_n$ are bounded in $L^3$, the square of the remainder in \eqref{nulllag} can be estimated by
\begin{align}\notag
\e_n^2\|\nabla_\bot u_n \|_{L^2}^2\| \nabla_\bot^2 u_n \|_{L^2}^2 &=  \e_n^2\|\nabla_\bot u_n \|_{L^2}^2\int_\Omega \left[\lambda_{1,n}^2 + \lambda_{2,n}^2 + 2\lambda_{1,n} \lambda_{2,n} - 2 \lambda_{1,n} \lambda_{2,n}\right] \, d\ex \\ \notag
&\leq C\e_n^2\int_\Omega (\Delta_\bot u_n)^2 \,d\ex + C\e_n^2\left| 2\int_\Omega \overline{K}_n \,d\ex \right|\\ \label{tildeebound}
&\to 0.
\end{align}
}Combining  \eqref{jkcalc2}-\eqref{tildeebound}, we conclude that
{\color{black}\begin{align}\notag
    \mathcal{E}_{\e_n}(u_n) &\geq \frac{1}{2}\int_{A}\left[ 
\frac{1}{\varepsilon_n }\left( \partial _{z}u_n-\frac{1}{2}\left\vert \nabla
_{\bot }u_n\right\vert ^{2}\right) ^{2}+\varepsilon_n \left(\lambda_{1,n} - \lambda_{2,n} \right)^2\right] d\mathbf{x}- {2\e_n\|\nabla \varphi \|_{L^\infty}\|\nabla_\bot u_n\|_{L^2} \|\nabla_\bot^2 u_n\|_{L^2}} \\ \label{bpsag}
&\geq|Iu_n|(A) - {\mathrm{o}(1)}.
\end{align}
}Since the limit inferior of the energies is finite, we can appeal to \Cref{repformula}.(iii) to find that $u\in AG^{3D}(A)$ and
\begin{equation*}
    \liminf_{n\to \infty}\mathcal{E}_{\e_n}(u_n) \geq \liminf_{n\to \infty}|I u_n|(A) \geq |Iu|(A).
\end{equation*}
An exhaustion argument gives $u\in \ag$ and \eqref{liminflapl}. The fact that $u\in \ago$ follows from
\begin{equation}
    \int_\Omega \left( \dz u_n - \frac{|\nabla_\bot u_n|^2}{2}\right)^2 d\ex \leq \e_n \mathcal{E}_{\e_n}(u_n) \to 0.
\end{equation}}
\end{proof}

\section{The Upper Bound}\label{sec:upbd}
In this section we show that the lower bound \Cref{3dlbd} is sharp when $u\in \ago$ and $\nabla u \in (BV\cap L^\infty)(\Omega)$ by means of a construction, so that we have matching upper and lower bounds. Combined with the lower bound, this allows us to conclude that under reasonable assumptions, equipartition of energy in \eqref{3dnonlinear} is optimal.

\begin{theorem}\label{3dupbd}
Let $u\in \ago \cap W^{1,\infty}(\Omega)$ and $\nabla u \in BV(\Omega;\mathbb{R}^3)$. Then there exists a sequence $\{u_\e \}\subset C^2(\Omega)$ such that 
\begin{equation}\label{w1p}
    u_\e \to u \textit{ in $W^{1,p}(\Omega)$ for all $1\leq p < \infty$}
\end{equation}
and
\begin{equation}\label{3dc}
  { \mathcal{E}_\e(u_\e)} \to \int_{J_{\nabla u}} \frac{|\nabla_\bot u^+ - \nabla_\bot u^-|^4}{12|\nabla u^+ - \nabla u^-|}\, d\mathcal{H}^2.
\end{equation}
\end{theorem}
The proof of \Cref{3dupbd} consists of two steps. In \Cref{cubeprop}, we show that on a cube with jump set parallel to one of the faces, the sequence of one-dimensional competitors with constant gradient in the direction parallel to the jump set is asymptotically minimizing. Second, the cube construction can be leveraged to obtain the full upper bound \Cref{3dupbd} by using the results of \cite{Pol13}.\par
To formulate the problem on a cube, let us fix an orthonormal basis ${\color{black}\{\zeta_1,\zeta_2,\nu \}}$ of $\R^3$ and the set
\begin{equation*}
    C= \{\ex \in \R^3 : |\ex \cdot \nu| \leq \sfrac{1}{2}, |\ex \cdot \zeta_i|\leq \sfrac{1}{2} \textup{ for }i=1,2 \}.
\end{equation*}
Next, we choose boundary data that will be compatible with a limiting jump set $\{\ex\in C :\ex \cdot \nu=0 \}$. Let $m^+\neq m^-$ be such that
\begin{equation}\label{compat}
    m_3^\pm = \frac{1}{2}|m_\bot^\pm|^2\quad \textup{ and }\quad\nu\parallelsum (m^+ - m^-),
\end{equation}
and consider the class 
\begin{align*}\notag
    \mathcal{A}_C := \{u \in H^2: \nabla u = m^\pm \textup{ when }\ex &\cdot \nu = \pm \sfrac{1}{2}\textup{ and $\nabla u$ is 1-periodic in the $\zeta_1, \zeta_2$ directions}\}.
\end{align*}
Note that since $m^+ \neq m^-$, the first equation in \eqref{compat} enforces
\begin{align}\notag
 \nu_\bot \neq 0,
\end{align}
so that we can define the planar unit vectors
\begin{align}\notag
\xi = \frac{\nu_\bot}{|\nu_\bot|}\quad\textup{and}\quad \eta = (-\xi_2, \xi_1, 0).
\end{align}
The smaller set of 1D competitors is defined by
\begin{align*}
    \mathcal{A}_C^{1D} := \{u\in \mathcal{A}_C: \nabla u \cdot \zeta_i = m^+ \cdot \zeta_i = m^- \cdot \zeta_i \textup{ for }i=1,2 \}.
\end{align*}
We remark that due to the boundary conditions imposed on the class $\mathcal{A}_C$ and the identity $\overline{K} =\det \nabla_\bot^2 u =\nabla_\bot (\dx u \dyy u, - \dx u \dxy u)$, 
{\color{black}
\begingroup
\allowdisplaybreaks
\begin{align}\notag
    \int_C (\Delta_\bot u)^2 \, d\ex &- \int_C |\nabla_\bot^2 u| \, d\ex \\ \notag
    &= 2\int_C    \dxx u \dyy u - (\dxy u)^2 \, d\ex \\ \notag
    &= 2\int_{\partial C \cap \{|\ex\cdot \nu| = \sfrac{1}{2}\}}   \dx u \partial_{\tau_{1,2}} (\dy u) \,d\mathcal{H}^2 + 2\sum_{i=1}^2 \int_{\partial C \cap \{|\ex\cdot \zeta_i|=\sfrac{1}{2} \}}\dx u \partial_{\tau_{1,2}} (\dy u) \,d\mathcal{H}^2 \\ \label{detid}
    &= 0 .
\end{align}
\endgroup}
We set
\begin{align}\notag
    r_\e = \inf_{\mathcal{A}_C} \mathcal{E}_\e\quad\textup{and}\quad  r_\e^{1D}= \inf_{\mathcal{A}_C^{1D}} \mathcal{E}_\e.
\end{align}
\begin{proposition}\label{cubeprop}
For any $\e>0$,
\begin{equation}\label{3dcubeq}
\frac{|m_\bot^+-m_\bot^-|^4}{12|m^+ - m^-|} \leq r_\e \leq r_\e^{1D} \leq \frac{|m_\bot^+-m_\bot^-|^4}{12|m^+ - m^-|}+c_1 e^{-c_2\e}.
\end{equation}
The constants $c_1$ and $c_2$ depend only on $m^+$ and $m^-$.
\end{proposition}
\begin{proof}
The inequality $r_\e \leq r_\e^{1D}$ is immediate, since $\mathcal{A}_C^{1D} \subset \mathcal{A}_C$. Also, the inequality 
\begin{equation*}
   \frac{|m_\bot^+-m_\bot^-|^4}{12|m^+ - m^-|} \leq r_\e 
\end{equation*}
follows from \eqref{jkcalc}, \eqref{jumpcost}, and the boundary conditions for $u\in\mathcal{A}_C$. Indeed, since $$\nu_\bot\parallelsum(m^+-m^-)_{\color{black}\bot} \parallelsum \xi,$$
we have
\begingroup
\allowdisplaybreaks
\begin{align} \notag
  \frac{|m_\bot^+-m_\bot^-|^4}{12|m^+ - m^-|} &\underset{\eqref{jumpcost}}{=}\bigg| \int_{C\cap\{\ex \cdot \nu = \sfrac{1}{2}\}}\Sigma_{\xi\eta}(m^+) \cdot \nu\,d\mathcal{H}^2 - \int_{C\cap\{\ex \cdot \nu = -\sfrac{1}{2}\}} \Sigma_{\xi\eta}(m^-) \cdot \nu\,d\mathcal{H}^2\bigg| \\ \notag
   &\hspace{.23cm}=\left|\int_C\dive \Sigma_{\xi\eta}( \nabla u) \,d\ex\right|\\ \notag
    &\underset{\eqref{jkcalc}}{\leq} \int_C \frac{1}{2\e}\left(\dz u - \frac{|\nabla_\bot u|^2}{2} \right)^2 + \frac{\e}{2}\left(\dZZ u - \dEE u\right)^2\,d\ex \\ \notag
    & \hspace{.23cm}\leq \int_C \frac{1}{2\e}\left(\dz u - \frac{|\nabla_\bot u|^2}{2} \right)^2 + \frac{\e}{2}\left((\dZZ u)^2 +2\left(\dEZ u \right)^2 +(\dEE u)^2\right)d\ex \\ \notag
    &\qquad- \e\int_C \left(\dZZ u \dEE u - (\dEZ u)^2 \right)\,d\ex \\ \notag
    &\hspace{.35cm}{= \frac{1}{2}\int_C \frac{1}{\e}\left(\dz u - \frac{|\nabla_\bot u|^2}{2} \right)^2 + \e |\nabla_\bot^2 u|^2 \,d\ex -\e \int_{C}\det (\nabla^2_\bot u) \,d\ex}\\ \notag
    &\hspace{.08cm}\underset{\eqref{detid}}{=} \mathcal{E}_\e(u).
\end{align}
\endgroup
Finally, showing that
\begin{equation}\label{est}
    r_\e^{1D} \leq \frac{|m_\bot^+-m_\bot^-|^4}{12|m^+ - m^-|}+c_1 e^{-c_2\e}
\end{equation}
entails constructing a sequence $\{ \nabla u_\e \}$ such that each $\nabla u_\e$ is a function of ${\mathbf{x}}\cdot \nu$ and
\begin{equation}\notag
    \mathcal{E}_\e(u_\e) \leq \frac{|m_\bot^+-m_\bot^-|^4}{12|m^+ - m^-|}+c_1 e^{-c_2\e}.
\end{equation}
Since the steps of such a construction are standard in the calculus of variations, we outline the procedure and refer to \cite[Proposition 5.2]{NY20}, which contains a full proof in the 2D case, for some of the estimates. \par
Let 
$$
p= m^+ - m^-
$$
and $g$ be the solution to the initial value problem
\begin{equation}\label{ivp}
\begin{cases} g'(t) = \displaystyle\frac{|g p_3 + \mm_3 -(gp_1 + \mm_1 )^2/2-(gp_2 + \mm_2 )^2/2|}{|p_\bot \cdot \nu_\bot|}, \\
{\color{black}g(0)}=\sfrac{1}{2}.\end{cases}
\end{equation}
Note that the denominator $|p_\bot \cdot \nu_\bot|\neq 0$ since $m_3^\pm = \frac{1}{2}|m_\bot^\pm|^2$, $m^+ \neq m^-$ imply that $m_\bot^+\neq m_\bot^-$. One can check that $g$ exists for all time and approaches $1$ and $0$ exponentially as $t\to \pm \infty$ (see for example \cite[Equation (1.21)]{Ste88}). Consider the family of functions
\begin{equation}\notag
    g \left(\frac{\ex \cdot \nu}{\e} \right)p+ \mm = \nabla \left[\e|p| G\left(\frac{\ex \cdot \nu}{\e} \right)+ m^-\cdot \ex\right] =: \nabla w_\e(\ex),
\end{equation}
where $G$ is an antiderivative of $g$, on the infinite strip $\{|\ex\cdot \tau_i|\leq \sfrac{1}{2}:i=1,2 \}$. Let $\xi$ be the unit vector $\nu_\bot/|\nu_\bot|$. By direct calculation, we have for any $\e>0$
\begingroup
\allowdisplaybreaks
\begin{align}\notag
    \mathcal{E}&_\e (\nabla w_\e) \\ \notag
&= \int_{\{|\ex\cdot \tau_i|\leq \sfrac{1}{2}:i=1,2 \}} \frac{1}{\e} \left(\dz w_\e - \frac{1}{2}|\nabla_\bot w_\e|^2 \right)^2 +\e \left(\Delta_\bot w_\e \right)^2\,d\ex \\ \notag 
&= \int_{-\infty}^\infty \left[\frac{1}{\e}\left(g\left(\frac{t}{\e} \right)p_3+ \mm_3  - \frac{\left(g\left(\frac{t}{\e} \right)p_1+\mm_1\right)^2}{2}-\frac{\left(g\left(\frac{t}{\e} \right)p_2+\mm_2\right)^2}{2}\right)^2+\e g'\left(\frac{t}{\e}\right)^2\frac{(p_\bot \cdot \nu_\bot)^2}{\e^2} \right] \,dt \\ \notag
&\hspace{-.15cm}\underset{\eqref{ivp}}{=} \left|\int_{-\infty}^\infty \left(g\left(\frac{t}{\e} \right)p_3+ \mm_3  - \frac{\left(g\left(\frac{t}{\e} \right)p_1+\mm_1\right)^2}{2}-\frac{\left(g\left(\frac{t}{\e} \right)p_2+\mm_2\right)^2}{2}\right)g'\left(\frac{t}{\e}\right)\frac{(p_\bot \cdot \nu_\bot)}{\e}\,dt\right|\\ \notag
&= \left|\int_{\{|\ex\cdot \tau_i|\leq \sfrac{1}{2}:i=1,2 \}} \left(\dz w_\e - \frac{|\nabla_\bot w_\e|^2}{2} \right)\dZZ w_\e \,d\ex \right|\\ \notag
&= \left|\int_{\{|\ex\cdot \tau_i|\leq \sfrac{1}{2}:i=1,2 \}} \left(\dz w_\e - \frac{|\nabla_\bot w_\e|^2}{2} \right)(\dZZ w_\e - \dEE w_\e) \,d\ex \right|\\ \notag
&\hspace{-.22cm}\underset{\eqref{jkcalc}}{=} \left|\int_{\{|\ex\cdot \tau_i|\leq \sfrac{1}{2}:i=1,2 \}} \dive \Sig(\nabla w_\e) \,d\ex \right|\\ \notag
&= |(\Sig(m^+) - \Sig(m^-))\cdot \nu |
\end{align}
\endgroup
In \eqref{jumpdens} in the lower bound, we saw that this was equal to
\begin{equation}\label{space}
   \frac{|m_\bot^+-m_\bot^-|^4}{12|m^+ - m^-|}. 
\end{equation}
The sequence $\{\nabla u_\e \}$ is constructed by suitably truncating $\nabla w_\e$ so that it is in the class $\mathcal{A}_C^{1D}$. The estimate \eqref{est} follows from the exponential approach of $g$ to $0$ and $1$ combined with \eqref{space}; see \cite[Proposition 5.2]{NY20} for full details.
\end{proof}
To prove \Cref{3dupbd}, we appeal to a general theorem from \cite{Pol08}. The version applicable to this problem reads as follows.
\begin{theorem}\label{polthm}
(\cite[Theorem 1.2]{Pol08})
Let $\Omega\subset \mathbb{R}^3$ be a bounded $C^2$-domain and let 
$$
F(a,b) : \mathbb{R}^{3\times 3} \times \mathbb{R}^3  \to \mathbb{R}
$$
be a $C^1$ function satisfying $F \geq 0$. Let $u \in W^{1,{\infty}}(\Omega)$ be such that $\nabla u \in BV(\Omega;\mathbb{R}^3)$ and $F(0,\nabla u(x))=0$ a.e. in $\Omega$. Then there exists a family of functions $\{ u_\e \} \subset C^2(\mathbb{R}^3)$ satisfying
\begin{equation}\notag
u_\e \to u \textit{ in }W^{1,{p}}(\Omega) \textit{ for }1\leq p< \infty
\end{equation}
and
\begin{align}\notag
&\lim_{\e \to 0} \frac{1}{\e} \int_\Omega F(\e \nabla^2 u_\e, \nabla u_\e)\,dx \,dz \\ \notag
&=\int_{J_{\nabla u}} \inf_{r \in \mathcal{R}_{\chi(x,z),0}} \left\{ \int_{-\infty}^\infty F\left(-r'(t) \nu(x,z) \otimes \nu(x,z), r(t)\nu(x,z)+\nabla u^-(x,z)\right) \,dt\right\}d\mathcal{H}^2.
\end{align}
Here $\chi(x,z)$ is given by
\begin{equation}\notag
\chi(x,z) \nu(x,z) = \nabla u^+(x,z)-\nabla u^-(x,z),
\end{equation}
and
$$
\mathcal{R}_{\chi(x,z),0}:=\{r(t) \in C^1(\mathbb{R}): \exists L>0 \textit{ s.t. } r(t)= \chi(x,z)  \textit{ for }t\leq -L, r(t) = 0\textit{ for }t\geq L\}
.$$
\end{theorem}
\begin{proof}[Proof of \Cref{3dupbd}]
If we set 
\begin{equation*}
    F(a,b) = \frac{1}{2}\left(b_3 - \frac{|b_\bot|^2}{2}\right)^2 + \frac{1}{2}\sum_{i=1}^2 a_{ii}^2,
\end{equation*}
then
\begin{equation*}
    \frac{1}{\e}\int_\Omega F(\e\nabla^2 u_\e,\nabla u_\e) \,d\ex =\mathcal{ E}_\e(u_\e).
\end{equation*}
To evaluate the infimum in \Cref{polthm}, we can rescale and use \Cref{cubeprop} to see that it is 
\begin{equation*}
    \frac{|\nabla u_\bot^+ - \nabla u_\bot^-|^4}{12|\nabla u^+ - \nabla u^-|}.
\end{equation*}
This finishes the proof.
\end{proof}
\begin{remark}\label{bcremark}
A recovery sequence with specified boundary data for $u$ and $\nabla u$ could be constructed as in \cite[Section 6]{ConDeL07} or \cite[Theorem 1.1]{Pol07}.
\end{remark}

Finally, let us rephrase the theorems of the last two sections in terms of the original problem involving smectics: roughly speaking, equipartition of energy is optimal when the Gaussian curvature induced by the boundary conditions is not prohibitively large. More precisely, denoting by $\lambda_i$ the eigenvalues of $\nabla_\bot^2 u$, we have
\begin{align}\notag
\mathcal{E}_{\e}(u) &= \frac{1}{2}\int_\Omega\left[\e(\lambda_{1} + \lambda_{2})^2 + \frac{1}{\e}\left(\dz u - \frac{1}{2}|\nabla_\bot u|^2 \right)^2 \right]\,d\ex \\ \notag
&= \frac{1}{2}\int_\Omega\left[\e(\lambda_{1} - \lambda_{2})^2 + \frac{1}{\e}\left(\dz u - \frac{1}{2}|\nabla_\bot u|^2 \right)^2 \right]\,d\ex + 2\e\int_\Omega \lambda_{1}\lambda_{2} \, d\ex \\ \label{ourbps}
&= |Iu|(\Omega) +\frac{1}{2} \int_\Omega \left(\e^{\sfrac{1}{2}}|\lambda_{1} - \lambda_{2}|- \frac{1}{\e^{\sfrac{1}{2}}}\left|\dz u - \frac{1}{2}|\nabla_\bot u|^2 \right| \right)^2\, d\ex + 2\e\int_\Omega \overline{K} \, d\ex.
\end{align}
If $\e\int_\Omega \overline{K}\,d\ex$ is small compared to the energy $\mathcal{E}_\e(u)$, which by \eqref{control} can be enforced by choosing boundary conditions such that $\e\| \nabla_\bot u \|_{H^{\sfrac{1}{2}}(\partial \Omega)}^2 $ is small, then contribution of the curvature term is negligible. Thus the energy $\mathcal{E}_\e$ is minimized by minimizing $|Iu|$ among competitors saturating the perfect square. The matching upper bound demonstrates that this procedure is optimal in a reasonable range of situations. Furthermore, saturation of the perfect square entails
\begin{equation}\label{equiapp}
\int_\Omega \e(\lambda_1 - \lambda_2)^2 \,d\ex \approx \int_\Omega \frac{1}{\e} \left(\dz u - \frac{1}{2}|\nabla_\bot u|^2 \right)^2\,d\ex .
\end{equation}
Since $\e(\Delta_\bot u)^2$ and $\e(\lambda_{1} - \lambda_{2})^2$ differ by $4\e \overline{K}$, the assumption that the integral of the curvature is small and \eqref{equiapp} imply that
\begin{equation}\label{equiapp2}
\int_\Omega \e(\Delta_\bot u)^2 \,d\ex \approx \int_\Omega \frac{1}{\e} \left(\dz u - \frac{1}{2}|\nabla_\bot u|^2 \right)^2\,d\ex ,
\end{equation}
which is precisely the BPS equation squared and integrated over $\Omega$.

\begin{remark}
Our 1D ansatz  satisfies BPS equation $\left( \ref{3dbps}\right)$. Also, the condition that $\e\int_\Omega\overline{K}$ must be small for equipartition to be optimal coincides with the observation from \cite{SanKam03} that BPS solutions are not competitive when the curvature is very large, so that the result is qualitatively sharp in some sense. 
\end{remark}
{\color{black}\begin{remark}
Both the arguments for the lower and upper bound hold for the sequence of energies
\begin{align*}
\tilde {\mathcal E_{\e}}(u)
=\frac{1}{2}\int_\Omega \left[\frac{1}{\e}\left(\dz u - \frac{1}{2}|\nabla_\bot u|^2 \right)^2 +\e|\nabla_\bot^2 u|^2 \right]\,d\ex
\end{align*}
with trivial modifications. For lower bound, we only need to assume $\liminf \tilde {\mathcal E}_{\e_n}(u_n)$ is finite. 
\end{remark}
}

\section{Compactness}

\label{sec:cpt} The main result in this section is the compactness theorem.

\begin{theorem}
\label{compactness}Let $\Omega \subset \mathbb{R}^{3}$ be a bounded {\color{black} domain with $C^1$ boundary}, $%
\varepsilon _{n}\rightarrow 0,$ and $\left\{ u_{n}\right\} $ $\subset
H^{2}\left( \Omega \right) $ be a sequence of functions with uniformly bounded energies $\mathcal{E}_{\varepsilon _{n}}\left( u_{n}\right) $ such that $\left\Vert
\nabla u_{n}\right\Vert _{L^{p}(\Omega)}\leq C$ for some $p>6$ {\color{black}and $\left\Vert
\nabla u_{n}\right\Vert _{L^{2}(\partial \Omega)}\leq C$}. Assume also that $\Delta_{\bot}u_n \geq 0 $ {or $\Delta_\bot u_n \leq 0$} a.e. in $\Omega$. Then $\nabla u_{n}$ is precompact in $L^{q}\left( \Omega \right) $ for any $1\leq q<p$. 
\end{theorem}

Theorem \ref{compactness} is a direct corollary of the following stronger
proposition.

\begin{proposition}
\label{gcompact}Let $\Omega \subset \mathbb{R}^{3}$ be a bounded {\color{black} domain with $C^1$ boundary} and $%
\left\{ u_{n}\right\} $ $\subset H^{2}\left( \Omega \right) $ be a sequence
of functions satisfying 

\begin{equation}
  {\color{black} \left\Vert
\nabla u_{n}\right\Vert _{L^{\color{black}p}(\Omega)}\leq C, \text{ for some } p>6},
\label{Linfinitybd}
\end{equation}

\begin{equation}
\| \nabla u_n \|_{L^2(\partial \Omega)} \leq C, \label{boundarybd}
\end{equation}

\begin{equation}
\partial _{z}u_{n}-\frac{1}{2}\left\vert \nabla _{\bot }u_{n}\right\vert
^{2}\rightarrow 0\text{ strongly in }L^{2}\left( \Omega \right) ,
\label{L2bound}
\end{equation}
and %
\begin{equation}
\left\vert \partial _{z}u_{n}-\frac{1}{2}\left\vert \nabla _{\bot
}u_{n}\right\vert ^{2}\right\vert \left\vert \Delta _{\bot }u_{n}\right\vert 
\text{ is bounded in }L^{1}\left( \Omega \right) .  \label{L1bound}
\end{equation}%
{\color{black}If in addition}
\begin{equation}
\Delta_{\bot}u_n \geq 0  \text{ a.e. in }\Omega \text{ or  }\Delta_{\bot}u_n \leq 0  \text{ a.e. in }\Omega ,\label{add}
\end{equation}
then $\left( \nabla _{\bot }u_{n}\right) $ is precompact in $L^{q}\left(
\Omega \right)$ for $1\leq q <p$.
\end{proposition}

\bigskip

We first prove a lemma used in the proof of \Cref{gcompact}.
\begin{lemma}
\label{divlemma}Under the assumptions $\left( %
\ref{Linfinitybd}\right)-\left(\ref{add}\right)$, $\dive B_{n}$ is {\color{black} relatively } compact in $H^{-1}\left(
\Omega \right) ,$ where 
\begin{equation*}
B_{n}=\left( -\frac{\nabla _{\bot }u_{n}}{2}\left\vert \nabla _{\bot
}u_{n}\right\vert ^{2},\frac{1}{2}\left\vert \nabla _{\bot }u_{n}\right\vert
^{2}\right) .
\end{equation*}
\end{lemma}

\begin{proof}
We prove the Lemma when  $u_{n}$ is smooth and general case follows by
approximating. By $\left( \ref{L2bound}\right) ,$%
\begin{eqnarray}
&&\partial _{z}\left( \partial _{x}u_{n}\right) -\partial _{x}\left( \frac{1%
}{2}\left\vert \nabla _{\bot }u_{n}\right\vert ^{2}\right)   \label{xeq} \\
&=&\partial _{x}\left( \partial _{z}u_{n}-\frac{1}{2}\left\vert \nabla
_{\bot }u_{n}\right\vert ^{2}\right) \rightarrow 0\text{ strongly in }%
H^{-1}\left( \Omega \right) , \quad\text{and} \notag \\
&&\partial _{z}\left( \partial _{y}u_{n}\right) -\partial _{y}\left( \frac{1%
}{2}\left\vert \nabla _{\bot }u_{n}\right\vert ^{2}\right)   \label{yeq} \\
&=&\partial _{y}\left( \partial _{z}u_{n}-\frac{1}{2}\left\vert \nabla
_{\bot }u_{n}\right\vert ^{2}\right) \rightarrow 0\text{ strongly in }%
H^{-1}\left( \Omega \right) .  \notag
\end{eqnarray}

Multiplying  $\left( \ref{xeq}\right) $ by $\partial _{x}u_{n}$ and $\left( %
\ref{yeq}\right) $ by $\partial _{y}u_{n}$ then summing,  we have 
\begin{eqnarray}
&&\partial _{z}\left( \frac{1}{2}\left\vert \nabla _{\bot }u_{n}\right\vert
^{2}\right) -\partial _{x}\left( \frac{\partial _{x}u_{n}}{2}\left\vert
\nabla _{\bot }u_{n}\right\vert ^{2}\right) -\partial _{y}\left( \frac{%
\partial _{y}u_{n}}{2}\left\vert \nabla _{\bot }u_{n}\right\vert ^{2}\right) 
\notag \\
&=&\partial _{x}\left( \partial _{x}u_{n}\left( \partial _{z}u_{n}-\frac{1}{2%
}\left\vert \nabla _{\bot }u_{n}\right\vert ^{2}\right) \right) +\partial
_{y}\left( \partial _{y}u_{n}\left( \partial _{z}u_{n}-\frac{1}{2}\left\vert
\nabla _{\bot }u_{n}\right\vert ^{2}\right) \right)   \notag \\
&&-\left( \partial _{z}u_{n}-\frac{1}{2}\left\vert \nabla _{\bot
}u_{n}\right\vert ^{2}\right) \Delta _{\bot }u_{n}-\frac{1}{2}\left\vert
\nabla _{\bot }u_{n}\right\vert ^{2}\Delta _{\bot }u_{n}  \notag \\
&=&I+II+III.  \label{divBn}
\end{eqnarray}%
Here 
\begin{equation*}
I=\partial _{x}\left( \partial _{x}u_{n}\left( \partial _{z}u_{n}-\frac{1}{2}%
\left\vert \nabla _{\bot }u_{n}\right\vert ^{2}\right) \right) +\partial
_{y}\left( \partial _{y}u_{n}\left( \partial _{z}u_{n}-\frac{1}{2}\left\vert
\nabla _{\bot }u_{n}\right\vert ^{2}\right) \right) {\color {black} \rightarrow 0 }
\end{equation*}%
 in $ {\color{black}W^{-1,\frac{2p}{p+2}}}\left( \Omega \right) $ {\color{black} up to a subsequence}, 
\begin{equation*}
II=-\left( \partial _{z}u_{n}-\frac{1}{2}\left\vert \nabla _{\bot
}u_{n}\right\vert ^{2}\right) \Delta _{\bot }u_{n}
\end{equation*}%
is bounded in $L^{1}(\Omega)$, and 
\begin{equation*}
III=-\frac{1}{2}\left\vert \nabla _{\bot }u_{n}\right\vert ^{2}\Delta _{\bot
}u_{n} .
\end{equation*}%
If $III$ is bounded in $\mathcal{M}\left( \Omega
\right) $, the space of measures, then the right hand side of $\left( \ref%
{divBn}\right) $ is the sum of a term {\color{black} relatively } compact in $W^{-1,\frac{2p}{p+2}}(\Omega)$ and a term bounded
in $\mathcal{M}\left( \Omega \right)$, so that by the embedding theorem, the right hand side of $\left( \ref{divBn}\right) $ is {\color{black} relatively} compact in $%
W^{-1,r}\left( \Omega \right) $ for some  $1\leq r<2$. On the other hand,
assumption $\left( \ref{Linfinitybd}\right) $ implies the left hand side of $%
\left( \ref{divBn}\right) $ is bounded in $W^{-1,{\color{black}\frac{p}{3}}} \left( \Omega
\right) $. {\color{black} Relative } compactness of $\dive B_{n}$ in $H^{-1}(\Omega)$ follows from interpolation.\par 
To finish the proof, we show $III$ is bounded in $\mathcal{M}\left( \Omega
\right)$ {\color{black} under the additional assumption \eqref{add}}. Rewrite $\left( \ref{divBn}\right) $ as %
\begin{equation}
\partial _{z}\left( \eta \left( u_{n}\right) \right) -\dive_{\bot
}\left( F\left( u_{n}\right) \right) +G\left( u_{n}\right) =\mu _{n},
\label{entropy}
\end{equation}%
where%
\begin{eqnarray*}
\eta \left( u_{n}\right)  &=&\frac{1}{2}\left\vert \nabla _{\bot
}u_{n}\right\vert ^{2},\text{ }F\left( u_{n}\right) =\left( \partial
_{x}u_{n}\partial _{z}u_{n},\partial _{y}u_{n}\partial _{z}u_{n}\right) , \\
G\left( u_{n}\right)  &=&\left( \partial _{z}u_{n}-\frac{1}{2}\left\vert
\nabla _{\bot }u_{n}\right\vert ^{2}\right) \Delta _{\bot }u_{n}, \\
\mu _{n} &=&-\frac{1}{2}\left\vert \nabla _{\bot }u_{n}\right\vert
^{2}\Delta _{\bot }u_{n}.
\end{eqnarray*}%
Let $\Omega^{-}=\{\ex \in \Omega: \mu _{n}(\ex)\leq 0\}$  and $\Omega^{+}=\{\ex \in \Omega: \mu _{n}(\ex)\geq 0\}$. Since (\ref{add}) holds, then $\Omega=\Omega^{-}$ or $\Omega=\Omega^{+}$. Assume $\Omega=\Omega^{-}$ (the other case can be proved similarly). Integrating $\left( %
\ref{entropy}\right) $ over $\Omega$ and using the divergence theorem on the first two terms yields

\begin{eqnarray*}
\int_{\Omega }-\mu _{n}dx\,dy\,dz &=&\int_{\partial\Omega}-\eta \left( u_{n}\right)\nu_3 \,d\mathcal{H}^2+\int_{\partial\Omega} F\left(
u_{n}\right)\cdot \nu_\bot \,d\mathcal{H}^2  -\int_\Omega G\left( u_{n}\right) d\ex \\
&\leq & \|\nabla_\bot u\|^2_{L^2(\partial \Omega)}+ \|\dz u\|_{L^2(\partial \Omega)}\|\nabla_\bot u\|_{L^2(\partial \Omega)}+ \mathcal{E}_{\e_n}(u_n)\\
&\leq& C
\end{eqnarray*} for some constant $C$ depending on the energy bound and $\| \nabla u_n \|_{L^2(\partial \Omega)}$. Therefore 
${-\mu_{n}}\mathcal{L}^3\mres \Omega$ is bounded in $\mathcal{M}\left(\Omega \right) $.
\end{proof}
{\color{black}
\begin{remark}

\begin{itemize}
\item[]
\item A special case satisfying (\ref{Linfinitybd}) is  $\|\nabla u_n\|_{L^{\infty}(\Omega)} \leq C$.
\item {\color{black} In the proof of Lemma \ref{divlemma}, the convergence result in \eqref{xeq} and \eqref{yeq} has not been used; however, it is used below in the proof of Proposition \ref{gcompact}.}
\end{itemize}
\end{remark}
}

\begin{proof}[Proof of Proposition \ref{gcompact}]
Set
\begin{equation*}
E_{n}=\left( \nabla _{\bot }u_{n},\frac{1}{2}\left\vert \nabla _{\bot
}u_{n}\right\vert ^{2}\right) \textup{  and  }B_{n}=\left( -\frac{\nabla _{\bot
}u_{n}}{2}\left\vert \nabla _{\bot }u_{n}\right\vert ^{2},\frac{1}{2}%
\left\vert \nabla _{\bot }u_{n}\right\vert ^{2}\right) .
\end{equation*}%
Lemma \ref{divlemma} together with $\left( \ref{xeq}\right)$ and $\left( \ref%
{yeq}\right) $ implies 
\begin{equation*}
\mathrm{curl}\, E_{n}\text{ and }\mathrm{div}\, B_{n}\text{ are {\color{black} relatively} compact in }%
H^{-1}\left( \Omega \right) .
\end{equation*}%
If $E_{n} \rightharpoonup E_{\infty },$ $B_{n} \rightharpoonup B_{\infty }$ in 
$L^{\color{black} 2 }\left( \Omega \right)$, then Tartar-Murat's div-curl Lemma applied to $%
E_n$ and $B_n$ yields
\begin{equation}
E_{n}\cdot B_{n}\overset{}{\rightharpoonup }E_{\infty }\cdot B_{\infty }%
\text{ in }\mathcal{D}^{\prime }\left( \Omega \right) .\text{ }
\label{weaklimit}
\end{equation}%
We introduce the following notations for the weak limits in $L^{\color{black}r
}\left( \Omega \right) {\color{black} \text{ for } r>1}:$%
\begin{equation*}
\left( \partial _{x}u_{n}\right) ^{4} \rightharpoonup U_{4},\text{ }\left(
\partial _{y}u_{n}\right) ^{4} \rightharpoonup V_{4},\text{ }\left( \partial
_{x}u_{n}\right) ^{3} \rightharpoonup U_{3},\text{ }\left( \partial
_{y}u_{n}\right) ^{3}\rightharpoonup V_{3},
\end{equation*}%
\begin{equation*}
\left( \partial _{x}u_{n}\right) ^{2}\partial _{y}u_{n} \rightharpoonup
U_{21},\text{ }\partial _{x}u_{n}\left( \partial _{y}u_{n}\right)
^{2} \rightharpoonup U_{12},\text{ }\left( \partial _{x}u_{n}\right)
^{2} \rightharpoonup U_{2},\text{ }\left( \partial _{y}u_{n}\right)
^{2} \rightharpoonup V_{2},
\end{equation*}%
and%
\begin{equation*}
\left( \partial _{x}u_{n}\right) ^{2}\left( \partial _{y}u_{n}\right)
^{2} \rightharpoonup U_{22},\,\partial _{x}u_{n}\partial
_{y}u_{n} \rightharpoonup U_{11},\text{ }\partial _{x}u_{n}\rightharpoonup
U_{1},\text{ }\partial _{y}u_{n} \rightharpoonup V_{1}.
\end{equation*}%
{Here $r$ depends on $p$ and the term in question but is greater than $1$ for each.} Under these notations, $\left( \ref{weaklimit}\right) $ can be written as 
\begin{equation*}
-\frac{1}{4}\left\vert \nabla _{\bot }u_{n}\right\vert ^{4}\rightharpoonup
\left( U_{1,}V_{1},\frac{1}{2}\left( U_{2}+V_{2}\right) \right) \cdot \left(
-\frac{1}{2}U_{3}-\frac{1}{2}U_{12},-\frac{1}{2}U_{21}-\frac{1}{2}V_{3},%
\frac{1}{2}\left( U_{2}+V_{2}\right) \right) .
\end{equation*}%
From this it follows that
\begin{equation*}
-\frac{1}{4}\left( U_{4}+2U_{22}+V_{4}\right) =-\frac{1}{2}U_{1}U_{3}-\frac{1%
}{2}U_{1}U_{12}-\frac{1}{2}V_{1}U_{21}-\frac{1}{2}V_{1}V_{3}+\frac{1}{4}%
\left( U_{2}+V_{2}\right) ^{2},
\end{equation*}%
or, equivalently,%
\begin{equation}
0=U_{4}+2U_{22}+V_{4}-2U_{1}U_{3}-2U_{1}U_{12}-2V_{1}U_{21}-2V_{1}V_{3}+%
\left( U_{2}+V_{2}\right) ^{2}.  \label{wkidentity}
\end{equation}%
Next, we consider 
\begin{eqnarray*}
&&\left( \left( \partial _{x}u_{n}\right) ^{2}-\partial
_{x}u_{n}U_{1}+\left( \partial _{y}u_{n}\right) ^{2}-\partial
_{y}u_{n}V_{1}\right) ^{2} \\
&=&\left( \partial _{x}u_{n}\right) ^{2}\left( \partial
_{x}u_{n}-U_{1}\right) ^{2}+2\partial _{x}u_{n}\partial _{y}u_{n}\left(
\partial _{x}u_{n}-U_{1}\right) \left( \partial _{y}u_{n}-V_{1}\right)
+\left( \partial _{y}u_{n}\right) ^{2}\left( \partial _{y}u_{n}-V_{1}\right)
^{2} \\
&\rightharpoonup &U_{4}-2U_{1}U_{3}+U_{2}U_{1}^{2}+2\left(
U_{22}-U_{21}V_{1}-U_{1}U_{12}+U_{11}U_{1}V_{1}\right)
+V_{4}-2V_{1}V_{3}+V_{2}V_{1}^{2} \\
&\underset{\eqref{wkidentity}}{=}&-\left( U_{2}+V_{2}\right)
^{2}+U_{2}U_{1}^{2}+V_{2}V_{1}^{2}+2U_{11}U_{1}V_{1}.
\end{eqnarray*}%
Observe that
{\color{black}\begin{eqnarray}\notag
\left( \partial _{x}u_{n}U_{1}+\partial _{y}u_{n}V_{1}-U_2-V_2\right)
^{2}
&=&(U_2+V_2)^2-2\left(\partial _{x}u_{n}U_{1}+\partial _{y}u_{n}V_{1}\right)\left(U_2+V_2\right)+\left( \partial _{x}u_{n}U_{1}+\partial _{y}u_{n}V_{1}\right)
^{2}\notag \\
&\rightharpoonup&(U_2+V_2)^2-2\left(U_2+V_2\right)\left(U^2_{1}+V^2_{1}\right)+ U_{2}U_{1}^{2}+V_{2}V_{1}^{2}+2U_{11}U_{1}V_{1}\notag \\
&\leq&(U_2+V_2)^2-U_{2}U_{1}^{2}-V_{2}V_{1}^{2}-2U_{11}U_{1}V_{1},\label{lesstn}
\end{eqnarray}%
where the last inequality follows from 
\begin{equation*}
U_2V_1^2+V_2U_1^2\geq 2U_{11}U_1V_1,
\end{equation*}
which is a direct conclusion from weak limit
\begin{equation*}
\left(\partial _{x}u_{n}V_{1}-\partial _{y}u_{n}U_{1}\right)^2\rightharpoonup U_2V_1^2+V_2U_1^2-2U_{11}U_1V_1.
\end{equation*}
Thus}
\begin{equation*}
\left( \left( \partial _{x}u_{n}\right) ^{2}-\partial _{x}u_{n}U_{1}+\left(
\partial _{y}u_{n}\right) ^{2}-\partial _{y}u_{n}V_{1}\right)
^{2}\rightharpoonup -\left( U_{2}+V_{2}\right)
^{2}+U_{2}U_{1}^{2}+V_{2}V_{1}^{2}+2U_{11}U_{1}V_{1}\underset{\eqref{lesstn}}{\leq} 0.
\end{equation*}%
From this we conclude that
\begin{equation*}
\left( \partial _{x}u_{n}\right) ^{2}-\partial _{x}u_{n}U_{1}+\left(
\partial _{y}u_{n}\right) ^{2}-\partial _{y}u_{n}V_{1}\rightarrow 0,
\end{equation*}%
so that passing to the limit on the left side, we have 
\begin{equation*}
U_{2}-U_{1}^{2}+V_{2}-V_{1}^{2}=0.
\end{equation*}%
Since {\color{black} $F(s)=s^2$ is a convex function, by Lemma 2 in  \cite{Tar79},  we have } $U_{2}\geq U_{1}^{2},$ $V_{2}\geq V_{1}^{2}$. Thus it must be the case that
\begin{equation*}
U_{2}=U_{1}^{2},\text{ }V_{2}=V_{1}^{2};
\end{equation*}%
{\color{black}in other words, $\lim_{n\rightarrow \infty}||\nabla_{\bot}u_n||_{L^2}=||\lim_n \nabla_{\bot}u_n||_{L^2}$, together with the weak convergence of $\nabla_{\bot}u_n$, } the strong convergence of $\nabla _{\bot }u_{n}$ in $L^{2}$ follows.
\end{proof}

\bigskip

\begin{proof}[Proof of Theorem \ref{compactness}]
Boundedness of $\mathcal{E}%
_{\varepsilon _{n}}\left( u_{n}\right) $ implies $\left( \ref{L2bound}%
\right) $ and $\left( \ref{L1bound}\right) .$  By Proposition \ref{gcompact}, $\nabla _{\bot }u_{n}$ is
precompact in $L^{2}(\Omega),$ {and hence in $L^q(\Omega)$ for $1\leq q < p$ by the uniform $L^p$ bound.} Compactness of $\partial _{z}u_{n}$ follows from
uniform boundedness of $\left\Vert \nabla _{\bot }u_{n}\right\Vert
_{L^{\color{black}p }(\Omega)}$ {\color{black} for $p>6$} and the fact $\partial _{z}u_{n}-\frac{1}{2}\left\vert
\nabla _{\bot }u_{n}\right\vert ^{2}\rightarrow 0$ in $L^{2}.$
\end{proof}
{\color{black}
\begin{remark}
The additional assumptions \eqref{Linfinitybd}, \eqref{boundarybd}, \eqref{add} in Theorem \ref{compactness} and Proposition  \ref{gcompact} are used in the proof of relative compactness of $\dive B_n$ in Lemma \ref{divlemma}.  We comment that the assumption  \eqref{Linfinitybd} is physically justifiable since the  model is only valid in  the limit of small strains \cite{SanKam03}. Assumption \eqref{boundarybd} is less restrictive than $\e_n \int_{\Omega}\bar{K} \rightarrow 0$. We would like to remove the technical assumption \eqref{add} in future work. An alternative approach to handling compactness is to rewrite the problem in terms of the geometric formulation of Tartar's conjecture \cite{Tar83}. Recall that the general question regarding upgrading weak convergence to strong convergence can be stated as follows: given a weakly convergent sequence of functions $z^{\varepsilon }:%
\mathbb{R}^{m}\rightarrow \mathbb{R}^{N}$ subject to linear differential
constraints of the form%
\begin{equation}
\sum_{j=0}^{m-1}A_{j}\partial _{j}z^{\varepsilon }=\varphi ^{\varepsilon },%
\text{ \ \ \ }A_{j}\text{ a }s\times N\text{ constant matrix,}
\label{linear}
\end{equation}%
and nonlinear algebraic constraints%
\begin{equation}
\left\{ z^{\varepsilon }\left( y\right) \right\} \subset M\text{ \ \ \ for
a.e. }y\in \mathbb{R}^{m},  \label{nonlinear}
\end{equation}%
where $M\subset \mathbb{R}^{N}$ is a subset, usually a manifold, what kind
of structure on $A_{j}$ and $M$ would suppress oscillations in $\left\{
z^{\varepsilon }\right\} $, so that $\{ z^\varepsilon\}$ contains a strongly convergent
subsequence? Tartar's conjecture can be expressed in terms of a geometric
condition. We introduce the oscillation variety%
\begin{equation*}
V=\left\{ \left( \xi ,\lambda \right) :\sum_{j=0}^{m-1}\xi _{j}A_{j}=0,\text{
}\xi \neq 0\right\} \subset \mathbb{R}^{m}\times \mathbb{R}^{N},
\end{equation*}%
and the wave cone, which is the projection of $V$ to $\mathbb{R}^{N}$:%
\begin{equation*}
\Lambda =P_{N}V=\left\{ \lambda :\exists \xi \neq 0,\text{ such that }\left(
\xi ,\lambda \right) \in V\right\} .
\end{equation*}%
Given any $a,$ let 
\begin{equation*}
\Lambda _{a}=a+\Lambda =\left\{ a+\lambda ,\lambda \in \Lambda \right\}
\end{equation*}%
be the translated cone. Tartar conjectured:

\begin{conjecture}
If the translated wave cone is separated from $M$ in the sense that 
\begin{equation*}
\Lambda _{a}\cap M=\left\{ a\right\} 
\end{equation*}%
for all $a,$ then the Young measure $\nu _{x}$ is a Dirac mass for almost every $x$, which implies the relative compactness in $L^p$
\end{conjecture}

For a sequence with bounded energy $\left( \ref%
{3denergy}\right) ,$ one may form a wave cone from \eqref{xeq}, \eqref{yeq}, and $\curl \nabla_{\bot} u=0$ and construct a constitutive manifold from suitable entropy conditions. Our initial observation \cite{NY22} shows the translated wave cone constructed this way is separated from the constitutive manifold. The final conclusion regarding compactness is still under investigation. 
\end{remark}
}

\textbf{ACKNOWLEDGEMENTS} We thank Robert V. Kohn for bringing this problem
into our interest.  M.N.'s research is supported by the NSF grant RTG-DMS 1840314. X.Y's research is supported by a
Research Excellence Grant and a CLAS Dean's summer research grant  from University of Connecticut. X.Y also thanks Fengbo Hang for helpful discussions on
Tartar's compensated compactness results. {\color{black} Both authors thank the anonymous referee for the thorough review and helpful comments.}

{\color{black}\textbf{DATA AVAILABILITY STATEMENT} Data sharing not applicable to this article as no datasets were generated or analysed during the current study.}

\FloatBarrier 
\bibliographystyle{siam}
\bibliography{references}

\end{document}